\documentclass[10pt]{amsart}
\usepackage{amsmath}
  \usepackage{enumerate}
\usepackage{hyperref}

\usepackage{thmtools, thm-restate} 
\usepackage{amssymb} 
\usepackage{stmaryrd} 
\usepackage{bbm} 

\usepackage{amsthm}
\usepackage[margin=1.5in]{geometry}

\theoremstyle{plain}
\newtheorem{theorem}{Theorem}[section]
\newtheorem{corollary}[theorem]{Corollary}

\newtheorem{lemma}[theorem]{Lemma}
\newtheorem{proposition}[theorem]{Proposition}
\newtheorem{conjecture}[theorem]{Conjecture}

\newtheorem{claim}[theorem]{Claim}

\theoremstyle{definition}
\newtheorem{definition}[theorem]{Definition}
\newtheorem{remark}[theorem]{Remark}

\newcommand{\ep}{\varepsilon}

\newtheorem{question}[theorem]{Question}
\newtheorem{example}[theorem]{Example}

\newcommand{\N}{\mathbb{N}}
\newcommand{\Z}{\mathbb{Z}}
\newcommand{\Q}{\mathbb{Q}}
\newcommand{\R}{\mathbb{R}}
\newcommand{\C}{\mathbb{C}}
\newcommand{\T}{\mathbb{T}}

\newcommand{\F}{\mathbb{F}}

\newcommand{\scrZ}{\mathcal{Z}}
\newcommand{\scrQ}{\mathcal{Q}}
\newcommand{\scrR}{\mathcal{R}}
\newcommand{\scrT}{\mathcal{T}}
\newcommand{\scrA}{\mathcal{A}}

\newcommand{\dual}{\mathbf{T}}

\renewcommand{\P}{\mathcal{P}}

\newcommand{\B}{\mathcal{B}}
\newcommand{\D}{\mathcal{D}}

\newcommand{\supp}[1]{\text{supp}\left( #1 \right)}

\newcommand{\acts}{\curvearrowright}

\newcommand{\tendsto}[1]{\xrightarrow[#1]{}}

\newcommand{\es}{\emptyset}

\newcommand{\ind}{\mathbbm{1}}

\newcommand{\id}{\text{id}}

\newcommand{\innprod}[2]{\left\langle #1, #2 \right\rangle}

\renewcommand{\hat}{\widehat}
\renewcommand{\tilde}{\widetilde}

\newcommand{\floor}[1]{\left[ #1 \right]}
\newcommand{\round}[1]{\left\llbracket #1 \right\rrbracket}
\newcommand{\fpart}[1]{\left\{ #1 \right\}}
\newcommand{\ceil}[1]{\left\lceil #1 \right\rceil}

\begin{document}

\title{Rigidity, weak mixing, and recurrence in abelian groups}
\author{Ethan M. Ackelsberg}
\address{Department of Mathematics, Ohio State University, Columbus, OH, USA 43210}
\email{ackelsberg.1@osu.edu}

\maketitle

\begin{abstract}
	The focus of this paper is the phenomenon of rigidity for measure-preserving actions
	of countable discrete abelian groups and its interactions with weak mixing and recurrence.
	We prove that results about $\Z$-actions extend to this setting:
	
	\begin{enumerate}[1.]
		\item	If $(a_n)$ is a rigidity sequence for an ergodic measure-preserving system,
	    		then it is a rigidity sequence for some weakly mixing system.
		\item	There exists a sequence $(r_n)$ such that every translate is both a rigidity sequence
			and a set of recurrence.
	\end{enumerate}
	
	\noindent The first of these results was shown for $\Z$-actions
	by Adams \cite{adams}, Fayad and Thouvenot \cite{ft}, and Badea and Grivaux \cite{bg}.
	The latter was established in $\Z$ by Griesmer \cite{griesmer}.
	While techniques for handling $\Z$-actions play a key role in our proofs,
	additional ideas must be introduced for dealing with groups with multiple generators.
	
	As an application of our results, we give several new constructions
	of rigidity sequences in torsion groups.
	Some of these are parallel to examples of rigidity sequences in $\Z$,
	while others exhibit new phenomena.
\end{abstract}


\section{Introduction}

In this paper, we study the phenomenon of \emph{rigidity} for measure-preserving actions of abelian groups.
We are particularly interested in how rigidity interacts with \emph{weak mixing} and \emph{recurrence}.
The interplay between these concepts has been examined in depth for the group of integers
(see, for example, \cite{bdjlr,adams,ft,bg,griesmer,fk}).
These properties have been explored separately for more general groups:
mild mixing, a notion complementary to rigidity, is explored in \cite{sw, schmidt};
general remarks on weak mixing can be found in \cite{dye, br};
and \cite{forrest} gives an overview of recurrence.
However, nontrivial relationships between them have not been established outside of the group of integers.
Moving to the context of general countable discrete abelian groups poses new challenges.
For instance, the group may be infinitely generated and contain torsion elements,
both of which can create difficulties in applying techniques from ergodic theory.
Nevertheless, in what follows, we show that several of the known results for $\Z$-actions
have natural extensions to the context of abelian groups.

We give special attention to $\bigoplus_{n=1}^{\infty}{\F_q}$, the direct sum of countably many copies of a finite field,
for constructing novel examples of rigidity sequences.
(We use $\F_q$ to denote the finite field as well as its additive group.)
On the one hand, we can view $\bigoplus_{n=1}^{\infty}{\F_q}$ as the additive group of the ring $\F_q[t]$,
which allows us to construct number-theoretically inspired rigidity sequences.
Namely, we construct rigidity sequences out of finite characteristic versions of
continued fraction expansions and Pisot--Vijayaraghavan numbers.
These objects are nontrivial to define in $\F_q[t]$,
but they can be considered as natural analogues to their counterparts in $\Z$.
On the other hand, $\bigoplus_{n=1}^{\infty}{\F_q}$ is an infinitely-generated torsion group,
and these properties can be exploited to construct rigidity sequences
that are dissimilar from any known constructions in $\Z$.

The general results relating rigidity to weak mixing and recurrence are explained below,
after providing the necessary background.


\subsection{Rigidity sequences}

Let $\Gamma$ be a countable discrete abelian group.
By a \emph{measure-preserving system} (or \emph{system} for short), we will mean a quadruple
$\left( X, \B, \mu,\right.\break\left. (T_g)_{g \in \Gamma} \right)$,
where $(X, \B, \mu)$ is a non-atomic Lebesgue probability space,
and $(T_g)_{g \in \Gamma}$ is an action of $\Gamma$ by measure-preserving transformations.
A sequence $(a_n)_{n \in \N}$ in $\Gamma$ is a \emph{rigidity sequence} for the system
$\left( X, \B, \mu, (T_g)_{g \in \Gamma} \right)$ if for every $f \in L^2(\mu)$, $\| f \circ T_{a_n} - f \|_2 \to 0$.
In this case, we also say that the system is \emph{rigid along the sequence} $(a_n)_{n \in \N}$,
or that $(a_n)_{n \in \N}$ is \emph{rigid} for $\left( X, \B, \mu, (T_g)_{g \in \Gamma} \right)$.


\subsection{Rigidity and weak mixing} \label{sec: rigid/WM}

Rigidity stands in contrast to mixing.
Whereas in rigid systems, the transformations $(T_g)_{g \in \Gamma}$ converge to the identity along some sequence,
actions with mixing properties produce asymptotic independence.
To make sense of asymptotic behavior, we need the notion of a \emph{F{\o}lner sequence}.
A sequence $(\Phi_N)_{N \in \N}$ of finite subsets of $\Gamma$ is a \emph{F{\o}lner sequence}
if for every $x \in \Gamma$,
\begin{align*}
	\frac{|(\Phi_N + x) \triangle \Phi_N|}{|\Phi_N|} \tendsto{N \to \infty} 0.
\end{align*}

\noindent Examples include the sequence of intervals $\Phi_N = \{1, \dots, N\}$ in $\Z$
and the sequence $\Phi_N = \bigoplus_{n=1}^N{\F_q}$ in $\bigoplus_{n=1}^{\infty}{\F_q}$.
Groups that admit a F{\o}lner sequence are called \emph{amenable},
and it is well-known that abelian groups are amenable.

Recall that a system $\left( X, \B, \mu, (T_g)_{g \in \Gamma} \right)$ is \emph{ergodic}
it there are no non-trivial invariant sets.
That is, if $T_gA = A$ for every $g \in \Gamma$, then $\mu(A) = 0$ or $\mu(A) = 1$.
By the mean ergodic theorem (see, e.g., \cite[Theorem 4.15]{et/dp}),
ergodicity can be expressed as a mixing property:
\begin{align*}
	\frac{1}{|\Phi_N|} \sum_{g \in \Phi_N}{\mu(A \cap T_gB)} \tendsto{N \to \infty} \mu(A)\mu(B)
\end{align*}

\noindent for every F{\o}lner sequence $(\Phi_N)_{N \in \N}$ in $\Gamma$ and all $A, B \in \B$.

A stronger property that we will deal with is \emph{weak mixing}:
\begin{align*}
	\frac{1}{|\Phi_N|} \sum_{g \in \Phi_N}{\left| \mu(A \cap T_gB) - \mu(A)\mu(B) \right|}
	 \tendsto{N \to \infty} 0
\end{align*}

\noindent for every F{\o}lner sequence $(\Phi_N)_{N \in \N}$ in $\Gamma$ and all $A, B \in \B$.
This is equivalent to the property:
there is an exceptional set $E \subseteq \Gamma$ of \emph{zero density}, meaning
\begin{align*}
	\lim_{N \to \infty}{\frac{|E \cap \Phi_N|}{|\Phi_N|}} = 0
\end{align*}

\noindent for every F{\o}lner sequence $(\Phi_N)_{N \in \N}$,
such that $\mu(A \cap T_gB) \to \mu(A) \mu(B)$ as $g \to \infty, g \notin E$.
Hence, rigidity sequences for weakly mixing systems, if they exist at all, must be relatively sparse:
they are (up to a finite initial segment) contained in the exceptional set $E$ of zero density.
Nevertheless, a generic (in the sense of category) measure-preserving automorphism
is simultaneously weakly mixing (see \cite{halmos}) and rigid (see \cite[Proposition 2.9]{bdjlr}).
So rigid weakly mixing $\Z$-actions are plentiful.

This does not, \emph{a priori}, say anything about which sequences arise
as rigidity sequences for weakly mixing systems.
However, Adams \cite{adams}, Fayad and Thouvenot \cite{ft}, and Badea and Grivaux \cite{bg}
showed that, in the case of $\Z$-actions, every possible rigidity sequence for an ergodic measure-preserving
transformation arises as a rigidity sequence for a weakly mixing transformation.
We extend this result to our setting:

\begin{restatable}{main}{rigidWM} \label{thm: rigidWM}
	Let $\Gamma$ be a countable discrete abelian group,
	and let $(a_n)_{n \in \N}$ be a sequence in $\Gamma$.
	If $(a_n)_{n \in \N}$ is a rigidity sequence for an ergodic measure-preserving system,
	then it is a rigidity sequence for some weakly mixing system.
\end{restatable}

Ideas from \cite{ft} and \cite{bg} turn out to be apt for generalization and are instrumental to our proof.
Because of this theorem, we will use the standalone term \emph{rigidity sequence}
to mean a sequence that is rigid for some weakly mixing system.
Theorem \ref{thm: rigidWM} is an extremely useful tool for constructing examples of rigidity sequences.
We explore this in Section \ref{sec: examples}.

One weakness of our notion of rigidity used here is that ``multi-dimensional'' groups such as
$\Z^d$ will possess many rigidity sequences coming from lower-dimen\-sional subgroups.
As an illustration of this phenomenon, given an ergodic automorphism $S : X \to X$,
we can define a $\Z^d$-action by $T_{(n_1, \dots, n_d)} := S^{n_1}$.
So long as $S$ is rigid along $(a_n)_{n \in \N}$, the sequence
$( a_n, b^{(2)}_n, \dots, b^{(d)}_n )_{n \in \N}$ will be a rigidity sequence in $\Z^d$
for every collection of sequences $( b^{(i)}_n )_{n \in \N}$ in $\Z$.

A natural way of avoiding these degenerated examples is to impose one additional condition.
A measure-preserving system $\left( X, \B, \mu, (T_g)_{g \in \Gamma} \right)$ is \emph{free}
if for every $g \in \Gamma \setminus \{0\}$, $\mu(\{x \in X : T_gx = x\}) = 0$.
If $T$ is ergodic, then this is equivalent to the seemingly weaker condition that $T$ is \emph{faithful}:
$\mu(\{x \in X : T_gx \ne x\}) > 0$ for every $g \in \Gamma$.
This is because the fixed points of a given automorphism $T_g$ are $T$-invariant.
As a consequence, every ergodic $\Z$-action on a Lebesgue space is automatically free,
so this assumption is superfluous in $\Z$.
However, the $\Z^d$-action above is ergodic but not free, and this is what produces the degeneration.
We can extend Theorem \ref{thm: rigidWM} to free actions as follows:

\begin{restatable}{main}{free} \label{thm: free}
	Let $\Gamma$ be a countable discrete abelian group.
	If $(a_n)_{n \in \N}$ is a rigidity sequence for a free ergodic measure-preserving system,
	then it is a rigidity sequence for some free weakly mixing system.
\end{restatable}


\subsection{Rigidity and recurrence}

In light of Theorem \ref{thm: rigidWM}, it is natural to ask what can be said about rigidity
for systems that do not exhibit weak mixing.
In particular, is every rigidity sequence rigid for a system with \emph{discrete spectrum}?
(Readers unfamiliar with discrete spectrum systems can turn to Section \ref{sec: DS/WM}
for a definition and several equivalent characterizations.
Included therein is a brief discussion of how discrete spectrum can be considered as a
complementary property to weak mixing.)
This was posed as an open question in \cite{bdjlr}
and answered negatively by Fayad and Kanigowski \cite{fk} in $\Z$.
A different construction due to Griesmer \cite{griesmer} gives an even stronger result
involving the phenomenon of \emph{recurrence}.
One advantage of Griesmer's approach is that it is generalizable to our setting.
To state the result, we need two more definitions.

\begin{restatable}{definition}{SetOfRec}
	Let $\Gamma$ be a countable discrete abelian group.
	A set $R \subseteq \Gamma$ is a \emph{set of recurrence}
	if for every measure-preserving system $\left( X, \B, \mu, (T_g)_{g \in \Gamma} \right)$
	and every $A \in \B$ with $\mu(A) > 0$,
	there is an $r \in R \setminus \{0\}$ such that $\mu \left( A \cap T_r^{-1}A \right) > 0$.
\end{restatable}

\begin{definition}
	Let $\Gamma$ be a countable discrete abelian group.
	A sequence $(r_n)_{n \in \N}$ in $\Gamma$ is \emph{rigid-recurrent} if $(r_n)_{n \in \N}$ is a rigidity sequence
	and $R = \{r_n : n \in \N\}$ is a set of recurrence.
\end{definition}

Using key insights from \cite{griesmer}, we prove the following (c.f. \cite[Theorem 2.1]{griesmer}):

\begin{restatable}{main}{RigRec} \label{thm: rigid-recurrent}
	Let $\Gamma$ be a countable discrete abelian group.
	There exists a sequence $(r_n)_{n \in \N}$ in $\Gamma$ such that, for every $t \in \Gamma$,
	$(r_n - t)_{n \in \N}$ is a rigid-recurrent sequence.
\end{restatable}

One consequence of Theorems \ref{thm: rigid-recurrent} is the following complement
to Theorems \ref{thm: rigidWM}, which we prove in Section \ref{sec: rigid rec}:

\begin{corollary} \label{cor: WM not DS}
	Let $\Gamma$ be a countable discrete abelian group.
	There is a sequence $(a_n)_{n \in \N}$ in $\Gamma$ such that $(a_n)_{n \in \N}$ is a rigidity sequence
	for a weakly mixing system but not for any ergodic system with discrete spectrum.
\end{corollary}

\begin{remark}
	Though it is not stated explicitly,
	Corollary \ref{cor: WM not DS} also follows from the work of Badea, Grivaux, and Matheron in \cite{bgm}.
	In particular, they establish, for every discrete abelian group,
	the existence of a rigidity sequence that is dense in the Bohr topology
	(see \cite[Section 5b.2]{bgm}).
	This is a weaker result than Theorem \ref{thm: rigid-recurrent} but sufficient to prove Corollary \ref{cor: WM not DS}.
\end{remark}

One may ask whether a version of Theorem \ref{thm: rigid-recurrent} holds if we impose a freeness condition.
Call a sequence $(r_n)_{n \in \N}$ in $\Gamma$ a \emph{freely rigid-recurrent sequence} if $(r_n)_{n \in \N}$
is a rigidity sequence for a free action of $\Gamma$ and $\{r_n : n \in \N\}$ is a set of recurrence.
In contrast to the situation above, where Theorem \ref{thm: rigidWM} extended directly to Theorem \ref{thm: free},
it is not possible in general to construct a sequence such that every translate is a freely rigid-recurrent sequence.
Indeed, we will show that for any rigidity sequence $(a_n)_{n \in \N}$ for a free action of the group
$(\Z/3\Z) \oplus \bigoplus_{n=1}^{\infty}{(\Z/2\Z)}$, the projection of $a_n$ onto $(\Z/3\Z)$
must be zero for all but finitely many $n$ (see Proposition \ref{prop: local obstruction} for a more general statement).
It follows that translates by elements of the form $t = (t_0; t_1, t_2, \dots)$ with $t_0 \in (\Z/3\Z) \setminus \{0\}$
can never be rigidity sequences for free actions of $(\Z/3\Z) \oplus \bigoplus_{n=1}^{\infty}{(\Z/2\Z)}$.

However, for this group and other similar groups,
we show that this obstacle can be overcome by restricting to translates coming from a finite index subgroup.
Recall that the \emph{exponent} of an abelian group $(\Gamma,+)$ is the smallest $n \in \N$
such that $nx = 0$ for every $x \in \Gamma$.
If no such $n$ exists, the group is said to have infinite exponent.
Every countable abelian group with finite exponent can be expressed as a direct sum of cyclic groups
\begin{align*}
	\Gamma := \bigoplus_{k=1}^M{(\Z/r_k\Z)^{m_k} }
	 \oplus \bigoplus_{j=1}^N{\left( \bigoplus_{n=1}^{\infty}{(\Z/q_j\Z)} \right)}
\end{align*}

\noindent with $q_1, \dots, q_N, r_1, \dots, r_M$ distinct prime powers and $m_1, \dots, m_M \in \N$
(see, e.g., \cite[Theorem 17.2]{fuchs}).
If all of the numbers $q_1, \dots, q_N$ are prime, then we have the following result:

\begin{theorem} \label{thm: free fin exp}
	Let $\Delta = \bigoplus_{j=1}^N{\left( \bigoplus_{n=1}^{\infty}{(\Z/p_j\Z)} \right)}$ with $p_1, \dots, p_N$ prime.
	Suppose $\Gamma = \Delta \oplus F$, where $F$ is a finite abelian group.
	Then there exists a sequence $(r_n)_{n \in \N}$ such that
	for every $s \in \Delta$, $(r_n - s)_{n \in \N}$ is a freely rigid-recurrent sequence.
\end{theorem}

More generally, we conjecture that the following is true:

\begin{conjecture} \label{conj: free rigid-recurrent}
	Let $\Gamma$ be a countable discrete abelian group.
	Then there exists a sequence $(r_n)_{n \in \N}$ and a finite index subgroup $\Delta \le \Gamma$ such that
	for every $s \in \Delta$, $(r_n - s)_{n \in \N}$ is a freely rigid-recurrent sequence.
\end{conjecture}

In Section \ref{sec: dense Kronecker}, we reduce Conjecture \ref{conj: free rigid-recurrent}
to a statement about the existence of certain thin sets in compact abelian groups.


\subsection{Structure of the paper}

In Section \ref{sec: background}, we introduce definitions and review basic facts
about weakly mixing and discrete spectrum measure-preserving systems.
We also show that, with the help of Gaussian systems,
rigidity can be characterized in terms of Fourier transforms of measures
on the Pontryagin dual group of $\Gamma$.
This will allow us to prove results about rigidity by constructing measures,
rather than producing measure-preserving systems directly.

From there, the paper is broken into two main parts.
In the first part, consisting of Sections \ref{sec: DS implies WM} and \ref{sec: examples},
we prove Theorem \ref{thm: rigidWM} and draw out several consequences.
Of particular note, we prove in Section \ref{sec: examples} that some sequences arising from number theory
(namely, denominators of convergents for continued fractions and powers of Pisot--Vijayaraghavan numbers)
are rigidity sequences for weakly mixing actions in $\Z$ and in $\F_q[t]$.
The fact that continued fraction denominators are rigidity sequences in $\Z$
was shown in \cite[Corollary 3.51]{bdjlr}.
Badea and Grivaux observed that a version of Theorem \ref{thm: rigidWM} in $\Z$
provides a more streamlined proof (see \cite[Section 6.3]{bg}),
We can similarly verify that our constructions are rigidity sequences using Theorem \ref{thm: rigidWM}.

The second part, consisting of Sections \ref{sec: recurrence}, \ref{sec: rigid rec}, and \ref{sec: free}
is devoted to proving Theorem \ref{thm: rigid-recurrent} and exploring Conjecture \ref{conj: free rigid-recurrent}.
Two of the main ingredients for the proof of Theorem \ref{thm: rigid-recurrent}
are discussed in Section \ref{sec: recurrence}:
(1) a well-known characterization of sets of recurrence in terms of sets of differences
via the Furstenberg correspondence principle, and
(2) the phenomenon of uniformity of recurrence.
We give a new proof of uniformity of recurrence for abelian groups in Section \ref{sec: unif rec proof}.
(This is of independent interest but is not directly needed for Theorem \ref{thm: rigid-recurrent}.)
Using these tools, we give the proof of Theorem \ref{thm: rigid-recurrent} in Section \ref{sec: rigid rec},
based on a generalization of a construction due to Griesmer \cite{griesmer}.
Finally, in Section \ref{sec: free}, we prove Theorem \ref{thm: free fin exp}
and reduce Conjecture \ref{conj: free rigid-recurrent} to a problem in harmonic analysis.


\section*{Acknowledgments}

Thanks to an anonymous referee for a number of helpful comments and suggestions that greatly improved the paper.
I thank my advisor, Vitaly Bergelson, for introducing me to several of the topics discussed in this paper
and asking a number of useful questions to push this work forward.
Thanks are also due to the participants in the Ergodic Theory and Combinatorial Number Theory Seminar
at Ohio State for being an active and encouraging audience for a talk on an incomplete version of this project.


\section{Preliminaries} \label{sec: background}

Throughout this section, let $\Gamma$ be a fixed countable discrete abelian group.


\subsection{Spectral measures}

Discrete spectrum and weak mixing are both \emph{spectral properties} of measure-preserving systems.
By this, we mean that these properties can be characterized in terms of the behavior
of the unitary action $f \mapsto f \circ T_g$ on $L^2(\mu)$.
We will adopt a standard abuse of notation and use $T_g$ to denote this unitary map on $L^2(\mu)$,
as well as the measure-preserving automorphism on $X$.

In dealing with spectral properties, it is useful to introduce \emph{spectral measures}.
These are measures that encapsulate the behavior of the orbit of a function $f \in L^2(\mu)$
under the action of $(T_g)_{g \in \Gamma}$ in terms of Fourier analysis.
Let $\dual := \hat{\Gamma}$ be the (compact abelian) Pontryagin dual group.
That is, $\dual$ is the group of homomorphisms $\chi : \Gamma \to \T$ under pointwise multiplication,
where we denote by $\T$ the unit circle in the complex plane.
The group $\dual$ is compact in the compact-open topology, since $\Gamma$ is countable and discrete.
We denote by $\P(\dual)$ the collection of Borel probability measures on $\dual$.
The subset of continuous\footnote{A measure $\sigma$ is \emph{continuous} if $\sigma(\{\chi\}) = 0$
for every $\chi \in \dual$.}
probability measures
will be denoted by $\P_c(\dual) \subseteq \P(\dual)$.
For $\sigma \in \P(\dual)$, we define the \emph{Fourier transform} $\hat{\sigma} : \Gamma \to \C$ by
\begin{align*}
	\hat{\sigma}(g) := \int_\dual{\chi(g)~d\sigma(\chi)}.
\end{align*}

The \emph{spectral measure} for a function $f \in L^2(\mu)$ is the positive measure
$\sigma_f$ on $\dual$ whose Fourier coefficients are given by
$\hat{\sigma}_f(g) = \int_X{\overline{f} \cdot T_gf~d\mu}$ for $g \in \Gamma$.
Such measures are guaranteed by Bochner's theorem,\footnote{Bochner's theorem states that if
$\varphi : \Gamma \to \C$ is a \emph{positive definite} function, meaning
$\sum_{i,j=1}^n{\overline{\xi}_i \varphi(g_i^{-1}g_j) \xi_j} \ge 0$
for all $n \in \N$, $g_1, \dots, g_n \in \Gamma$, and $\xi_1, \dots, \xi_n \in \C$,
then $\varphi$ is the Fourier transform of a positive measure on the dual group $\dual = \hat{\Gamma}$.
A proof of Bochner's theorem can be found in \cite[Theorem 30.3]{hr2} or \cite[Theorem 1.4.3]{rudin}.}
since the function $g \mapsto \int_X{\overline{f} \cdot T_gf~d\mu}$ is positive definite on $\Gamma$.
The \emph{maximal spectral type} of the system $\left( X, \B, \mu, (T_g)_{g \in \Gamma} \right)$
is a probability measure $\sigma \in \P(\dual)$ such that for every $f \in L^2(\mu)$,
the spectral measure $\sigma_f$ is absolutely continuous with respect to $\sigma$,
and if $\nu$ is any other measure with the same property, then $\sigma \ll \nu$.
Note that the maximal spectral type is only defined up to equivalence.
See \cite[Section 1]{kt} for more on the maximal spectral type.


\subsection{Discrete spectrum and weak mixing} \label{sec: DS/WM}

A function $f \in L^2(\mu)$ is an \emph{eigenfunction} for the system
$\left( X, \B, \mu, (T_g)_{g \in \Gamma} \right)$ if there is a function $\chi : \Gamma \to \C$
such that $T_gf = \chi(g) f$ for every $g \in \Gamma$.
Since $(T_g)_{g \in \Gamma}$ is a group action, $\chi$ is necessarily a character on $\Gamma$,
which we call an \emph{eigenvalue} or \emph{eigencharacter}.
We say the system $\left( X, \B, \mu, (T_g)_{g \in \Gamma} \right)$ has \emph{discrete spectrum}
if $L^2(\mu)$ is spanned by eigenfunctions.

Several other characterizations of discrete spectrum systems are given by the following:

\begin{proposition} \label{prop: DS equiv}
	Let $\left( X, \B, \mu, (T_g)_{g \in \Gamma} \right)$ be an ergodic measure-preserving system.
	The following are equivalent:
	\begin{enumerate}[(i)]
		\item	$T$ has discrete spectrum.
		\item	There is a compact abelian group $(Y, +)$ with Haar measure $\nu$
			on the Borel $\sigma$-algebra $\D = Borel(Y)$
			and an action $S : \Gamma \acts Y$ by group rotations $S_gy = y + \alpha_g$ with $\alpha_g \in Y$
			such that $\left( X, \B, \mu, (T_g)_{g \in \Gamma} \right)$
			and $\left( Y, \D, \nu, (S_g)_{g \in \Gamma} \right)$
			are isomorphic as measure-preserving systems;
		\item	For every $f \in L^2(\mu)$, the orbit $\{ T_gf : g \in \Gamma\}$ is totally bounded;
		\item	For every $f \in L^2(\mu)$, the spectral measure
			$\sigma_f$ is atomic.
		\item	The maximal spectral type is atomic.
	\end{enumerate}
\end{proposition}

For the details of these equivalences for $\Z$-actions, see \cite{walters, nadkarni};
the proofs easily extend to abelian groups.

Systems with discrete spectrum are also called \emph{compact} systems because of properties (ii) and (iii).
Complementary to discrete spectrum is the notion of weak mixing.
Recall from Section \ref{sec: rigid/WM} that a system $\left( X, \B, \mu, (T_g)_{g \in \Gamma} \right)$
is \emph{weakly mixing} if
\begin{align*}
	\frac{1}{|\Phi_N|} \sum_{g \in \Phi_N}{\left| \mu(A \cap T_gB) - \mu(A)\mu(B) \right|}
	 \tendsto{N \to \infty} 0
\end{align*}

\noindent for every F{\o}lner sequence $(\Phi_N)_{N \in \N}$ in $\Gamma$ and every $A, B \in \B$.
This, too, has many equivalent characterizations, some of which we give here:

\begin{proposition}
	Let $\left( X, \B, \mu, (T_g)_{g \in \Gamma} \right)$ be a measure-preserving system.
	The following are equivalent:
	\begin{enumerate}[(i)]
		\item $T$ is weakly mixing;
		\item	$\left( X \times X, \B \otimes \B, \mu \otimes \mu, (T_g \times T_g)_{g \in \Gamma} \right)$
			is ergodic;
		\item	$T$ has no non-constant compact fuctions:
			if $f \in L^2(\mu)$ and the orbit $\{T_gf : g \in \Gamma\}$
			is totally bounded, then $f$ is a constant function;
		\item	$T$ has no non-constant eigenfunctions;
		\item	For every $f_1, f_2 \in L^2(\mu)$ and every F{\o}lner sequence
			$(\Phi_N)_{N \in \N}$ in $\Gamma$,
						\begin{align*}
				\frac{1}{|\Phi_N|} \sum_{g \in \Phi_N}{
				\left| \int_X{f_1 \cdot T_gf_2~d\mu} - \int_X{f_1~d\mu} \int_X{f_2~d\mu} \right|}
				\tendsto{N \to \infty} 0;
			\end{align*}
			
		\item	For every $f \in L^2(\mu)$ with $\int{f~d\mu} = 0$,
			the spectral measure $\sigma_f$ is continuous.
		\item	$T$ is ergodic and the maximal spectral type is the sum of
			a continuous measure and an atom at $1 \in \dual$.
	\end{enumerate}
\end{proposition}

See \cite{br, nadkarni} for proofs.
To get property (vi) in this generality, an abelian version of Wiener's lemma must be used
in place of the classical Wiener's lemma as it appears in \cite{nadkarni}.

Weakly mixing systems are also known as systems with \emph{continuous spectrum} because of property (vi).
In fact, weakly mixing systems were first introduced in \cite{kvn} as
\emph{dynamical systems of continuous spectra}.


\subsection{Gaussian systems}

In order to complete the translation between dynamical properties and properties of measures,
we need to be able to construct measure-preserving $\Gamma$-systems from measures on $\dual$.
To this end, we review basic properties of \emph{Gaussian systems}.

Let $\Omega$ be the sequence space $\C^{\Gamma}$, and let $S_g : \Omega \to \Omega$ be the shift
$(S_g\omega)(h) = \omega(h-g)$.
Let $\pi_0 : \Omega \to \C$ be the function $\pi_0(\omega) := \omega(0)$.
An $S$-invariant probability measure $\nu$ on $\B_\Omega$, the Borel subsets of $\Omega$ with the product topology,
is called a \emph{Gaussian measure} if, for every $n \in \N$ and $g_1, \dots, g_n \in \Gamma$, the tuple
\begin{align*}
	\left( S_{g_1}\pi_0, \dots, S_{g_n}\pi_0 \right) : (\Omega, \B_{\Omega}, \nu) \to \C^n
\end{align*}

\noindent follows a Gaussian distribution.

A measure-preserving system $\left( X, \B, \mu, (T_g)_{g \in \Gamma} \right)$ is a \emph{Gaussian system}
if there is a Gaussian measure $\nu$ such that
$\left( X, \B, \mu, (T_g)_{g \in \Gamma} \right) \simeq \left( \Omega, \B_{\Omega}, \nu, (S_g)_{g \in \Gamma} \right)$.
A function $f \in L^2(\mu)$ is a \emph{Gaussian element} if $f$ maps to $\pi_0$ under this isomorphism.

\begin{theorem} \label{thm: unique Gauss}
	Let $\varphi : \Gamma \to \C$ be a positive definite function with $\varphi(0) = 1$.
	Then there exists a unique $S$-invariant symmetric Gaussian measure
	$\nu$ on $(\Omega, \B_{\Omega})$ with covariances determined by $\varphi$, i.e.
		\begin{align*}
		\int_{\Omega}{\pi_0~d\nu} = \int_{\Omega}{\pi_0^2~d\nu} = 0
	\end{align*}
	
	\noindent and
		\begin{align*}
		\int_{\Omega}{\overline{\pi_0} \cdot S_g\pi_0~d\nu} = \varphi(g)
	\end{align*}
	
	\noindent for every $g \in \Gamma$.
\end{theorem}
\begin{proof}
	The proof is essentially the same as in the real-valued case.
	We give a sketch of the proof for the complex case and refer the reader
	to \cite[Appendix C]{kechris} for additional details.
	
	Let $F \subseteq \Gamma$ be a finite set.
	There is a unique Gaussian measure $p_F$ on $\C^F$ satisfying
		\begin{align*}
		\int_{\C^F}{\omega(g)~dp_F} = \int_{\C^F}{\omega(g)^2~dp_F} = 0
	\end{align*}
	
	\noindent and
		\begin{align*}
		\int_{\C^F}{\overline{\omega(g)} \omega(h)~dp_F} = \varphi(g^{-1}h)
	\end{align*}
	
	\noindent for $g, h \in F$ (see \cite[Section 1.4]{janson}).
	By the Kolmogorov extension theorem, there is then a unique measure $\nu$ on $(\Omega, \B_{\Omega})$
	such that $(\pi_F)_*\nu = p_F$, where $\pi_F$ is the projection $\Omega \to \C^F$.
	The measure $\nu$ thus constructed is the desired measure.
\end{proof}

Using this theorem, we can construct a Gaussian system from any probability measure on $\dual$:

\begin{definition}
	Let $\sigma \in \P(\dual)$.
	The \emph{Gaussian system associated to $\sigma$} is the system
	$\left( \Omega, \B_{\Omega}, \nu_{\sigma}, (S_g)_{g \in \Gamma} \right)$,
	where $\nu_{\sigma}$ is the unique measure guaranteed by Theorem \ref{thm: unique Gauss}
	for the positive definite function $g \mapsto \hat{\sigma}(g)$.
\end{definition}

The dynamical properties of the Gaussian system associated to a measure $\sigma$
are well-understood in terms of properties of the measure.
We will make frequent use of the following two results.

\begin{proposition} \label{prop: WM Gaussian}
	Let $\sigma \in \P(\dual)$, and let $\left( X, \B, \mu, (T_g)_{g \in \Gamma} \right)$ be the Gaussian system
	associated to $\sigma$.
	The following are equivalent:
	
	\begin{enumerate}[(i)]
		\item	$T$ is ergodic;
		\item	$T$ is weakly mixing;
		\item	$\sigma$ is continuous.
	\end{enumerate}
\end{proposition}
\begin{proof}
	See \cite[Section 14.2]{cfs} for a proof in the case $\Gamma = \Z$.
	The general case of countable discrete abelian groups holds by the same method.
\end{proof}

\begin{lemma} \label{lem: rigidity Gaussian}
	Let $\sigma \in \P(\dual)$, and let $\left( X, \B, \mu, (T_g)_{g \in \Gamma} \right)$ be the Gaussian system
	associated to $\sigma$.
	A sequence $(a_n)_{n \in \N}$ is a rigidity sequence for $T$ if and only if $\hat{\sigma}(a_n) \to 1$.
\end{lemma}
\begin{proof}
	Let $f \in L^2(\mu)$ be a Gaussian element.
	By definition, the spectral measure $\sigma_f$ coincides with the measure $\sigma$.
	This immediately gives the equivalence: $T_{a_n}f \to f$ in $L^2(\mu)$ if and only $\hat{\sigma}(a_n) \to 1$.
	
	Now, if $(a_n)_{n \in \N}$ is a rigidity sequence, then clearly $T_{a_n}f \to f$ in $L^2(\mu)$.
	For the converse, note that the algebra of functions generated by $\{T_gf : g \in \Gamma\}$ is dense in $L^2(\mu)$.
	Hence, if $T_{a_n}f \to f$, then $T_{a_n}h \to h$ for every $h \in L^2(\mu)$.
\end{proof}


\subsection{Rigidity and Fourier analysis}

Using Gaussian systems, one can show that a sequence $(a_n)_{n \in \N}$
is a rigidity sequence for some weakly mixing system if and only if there is a
continuous probability measure $\sigma \in \P_c(\dual)$ such that $\hat{\sigma}(a_n) \to 1$.
This was observed for $\Z$-actions in \cite[Proposition 2.30]{bdjlr}.
Rigidity sequences for ergodic systems with discrete spectrum can also be characterized
in terms of the behavior of characters, an observation that is implicit in \cite{bg} in the context of $\Z$-actions.

\begin{lemma} \label{lem: rigid characters}
	A sequence $(a_n)_{n \in \N}$ in $\Gamma$ is
	\begin{enumerate}[(1)]
		\item	rigid for some ergodic system with discrete spectrum if and only if
			there are infinitely many characters $\chi \in \dual$ such that $\chi(a_n) \to 1$;
		\item	rigid for some weakly mixing system if and only if there is a measure $\sigma \in \P_c(\dual)$
			such that $\hat{\sigma}(a_n) \to 1$.
	\end{enumerate}
\end{lemma}

\begin{remark}
	Recall that we deal only with measure-preserving actions on non-atomic Lebesgue spaces.
	For rotations on finitely many points, it is possible to have a rigidity sequence $(a_n)_{n \in \N}$
	for which the set $\{\chi \in \dual : \chi(a_n) \to 1\}$ is finite.
	For example, the sequence $(2n)_{n \in \N}$ is a rigidity sequence for a rotation on two points (as a $\Z$-system),
	but $\lambda^{2n} \to 1$ only for $\lambda \in \{1, -1\}$.
	By assuming that we deal with non-atomic Lebesgue spaces, we rule out such trivial examples.
\end{remark}

\begin{proof}
	(1) Let $\left( X, \B, \mu, (T_g)_{g \in \Gamma} \right)$ be an ergodic measure-preserving system
	with discrete spectrum that is rigid along $(a_n)_{n \in \N}$.
	By Proposition \ref{prop: DS equiv}(ii), we may assume that $X$ is a compact abelian group
	and $(\alpha_g)_{g \in \Gamma}$ are elements of $X$ such that $T_gx = x + \alpha_g$.
	Note that $g \mapsto \alpha_g$ is a homomorphism.
	Hence, for any character $\lambda \in \hat{X}$, the function $\chi_{\lambda}(g) := \lambda(\alpha_g)$
	is a character on $\Gamma$.
	
	Now, for $\lambda \in \hat{X}$, we have $T_g\lambda = \chi_{\lambda}(g) \lambda$.
	Since $T$ is ergodic, it has no non-trivial invariant functions, so $\chi_{\lambda} = 1$
	if and only if $\lambda = 1$.
	That is, the map $\lambda \mapsto \chi_{\lambda}$ is injective,
	so $\left\{\chi_{\lambda} : \lambda \in \hat{X}\right\}$ is infinite.
	
	Finally, since $(a_n)_{n \in \N}$ is a rigidity sequence for $T$, we have
		\begin{align*}
		\left| \chi_{\lambda}(a_n) - 1 \right|= \left| \lambda(\alpha_{a_n}) - 1 \right|
		 = \left\| \lambda(x + \alpha_g) - \lambda(x) \right\|_2
		 = \left\| T_{a_n}\lambda - \lambda \right\|_2 \to 0.
	\end{align*}
	
	Conversely, suppose $C = \{\chi \in \dual : \chi(a_n) \to 1\}$ is infinite.
	Observe that $C$ is a subgroup of $\dual$.
	Let $\Lambda$ be a countably infinite subgroup of $C$ endowed with the discrete topology,
	and set $X := \hat{\Lambda}$.
	Equip $X$ with the Borel $\sigma$-algebra $\B$ and Haar probability measure $\mu$.
	Then $(X, \B, \mu)$ is a non-atomic Lebesgue space, and $\Gamma \acts (X, \B, \mu)$ by
	$T_gx = e_g \cdot x$, where $e_g(\lambda) := \lambda(g)$ for $\lambda \in \Lambda \subseteq \dual$.
	This is an action by group rotations, so it has discrete spectrum.
	Now, for every $\lambda \in \Lambda$, $(T_{a_n}x)(\lambda) = \lambda(a_n) x(\lambda) \to x(\lambda)$,
	so $T_{a_n}x \to x$ pointwise on $X$.
	It follows that $\|T_{a_n}f - f\|_2 \to 0$ for every $f \in L^2(\mu)$.
	
	\medskip
	
	(2) Let $\left( X, \B, \mu, (T_g)_{g \in \Gamma} \right)$ be a weakly mixing system
	such that $T_{a_n} \to \id$.
	Let $f \in L^2(\mu)$ with $\|f\| = 1$ and $\int_X{f~d\mu} = 0$.
	By Bochner's theorem, let $\sigma_f \in \P(\dual)$ be the spectral measure:
	$\hat{\sigma}_f(g) = \int_X{\overline{f} \cdot T_gf~d\mu}$.
	Since the action is weakly mixing, $\sigma_f \in \P_c(\dual)$.
	Moreover,
		\begin{align*}
		\hat{\sigma}_f(a_n) = \int_X{\overline{f} \cdot T_{a_n}f~d\mu} \to \int_X{|f|^2~d\mu} = \|f\|^2 = 1.
	\end{align*}
	
	Conversely, suppose $\sigma \in \P_c(\dual)$ and $\hat{\sigma}(a_n) \to 1$.
	By Lemma \ref{lem: rigidity Gaussian}, the Gaussian system associated to $\sigma$
	is rigid along $(a_n)_{n \in \N}$.
\end{proof}

Now Theorem \ref{thm: rigidWM} will follow from

\begin{theorem} \label{thm: DS implies WM}
	Let $(a_n)_{n \in \N}$ be a sequence in $\Gamma$, and suppose that the set
		\begin{align*}
		C := \{ \chi \in \dual : \chi(a_n) \to 1 \}
	\end{align*}
	
	\noindent is infinite.
	Then there is a continuous probability measure $\sigma \in \P_c(\dual)$
	such that $\hat{\sigma}(a_n) \to 1$.
\end{theorem}

This is a direct generalization of \cite[Corollary 2.5]{bg}.
We prove Theorem \ref{thm: DS implies WM} in Section \ref{sec: DS implies WM proof}.


\section{Proof of Theorem \ref{thm: rigidWM} and corollaries} \label{sec: DS implies WM}

The goal of this section is to prove Theorem \ref{thm: rigidWM} and explore some of its consequences,
including Theorem \ref{thm: free} and related topological results.
The proof follows the strategy used by Fayad--Thouvenot \cite{ft} in the case $\Gamma = \Z$
(see also \cite[Section 4]{bg}).
With some modifications, we are able to extend their result to the setting of general abelian groups.
As noted above, it suffices to prove Theorem \ref{thm: DS implies WM}, which we do now.


\subsection{Proof of Theorem \ref{thm: DS implies WM}} \label{sec: DS implies WM proof}

The basic outline of the proof is as follows.
We are given that $C = \{ \chi \in \dual : \chi(a_n) \to 1\}$ is infinite.
From this, we deduce that $\overline{C}$ is an (uncountable) compact subgroup of $\dual$.
We then construct a sequence of discrete measures sitting on $C$.
On the one hand, we need to pick the points using a Cantor-like construction so that any limiting measure is continuous.
On the other hand, we need to ensure that convergence to 1 along the sequence
$(a_n)_{n \in \N}$ happens sufficiently uniformly among the chosen points from $C$
so that the Fourier coefficients of the limiting measure converge to 1.

Now we proceed with the proof as outlined above.
By assumption, $C = \{ \chi \in \dual : \chi(a_n) \to 1 \}$ is infinite.
Moreover, $C$ is a subgroup of $\dual$.
Now, since $\dual$ is compact and $C$ is infinite, $C$ has a cluster point.
The (topological) group structure guarantees that every point of $C$ is a cluster point.
This is the key property we will exploit to construct a continuous probability measure
$\sigma \in \P_c(\dual)$ with $\hat{\sigma}(a_n) \to 1$.

We will construct a sequence of characters $(\chi_i)_{i \in \N}$ in $C$,
a sequence of nonnegative integers $(N_p)_{p \ge 0}$, and a family of open subsets
$\left( V_{p,r} \right)_{1 \le r \le 2^p, p \ge 0}$ of $\dual$ so that the measures
$\sigma_p = 2^{-p} \sum_{i=1}^{2^p}{\delta_{\chi_i}}$ have the following three properties for $p \ge 0$:
\begin{itemize}
	\item	For $0 \le j \le p - 1$ and $N_j \le n \le N_{j+1}$,
				\begin{align} \label{eq: near 1, small index}
			\int_\dual{\left| \chi(a_n) - 1 \right|~d\sigma_p(\chi)} & < 2^{-(j-1)}.
	\intertext{\item	For $n \ge N_p$,} \label{eq: near 1, large index}
			\int_\dual{\left| \chi(a_n) - 1 \right|~d\sigma_p(\chi)} & < 2^{-(p+1)}.
	\intertext{\item	The sets $V_{p,1}, \dots, V_{p,2^p}$ are mutually disjoint,
	and for $0 \le q \le p - 1$, $0 \le l \le 2^{p-q} - 1$, and $1 \le r \le 2^q$,}
		\label{eq: Cantor}
			\chi_{l \cdot 2^q + r} \in V_{p, l \cdot 2^q + r}
			 & \subseteq V_{q,r}.
		\end{align}
		
\end{itemize}

Properties \eqref{eq: near 1, small index} and \eqref{eq: near 1, large index} control the rate of
convergence of the characters along $(a_n)_{n \in \N}$ so that the Fourier coefficients of a limiting measure
will converge to 1 along $(a_n)_{n \in \N}$.
Property \eqref{eq: Cantor} ensures that the mass assigned
to a sufficiently small neighborhood of $\chi_r$ is bounded by $2^{-q}$ if $r \le 2^q$,
so the limiting measure is continuous.

Now we make these claims precise.
Suppose we have constructed such sequences, and let $\sigma$ be a weak-$^*$
limit point of the sequence of measures $(\sigma_p)_{p \ge 0}$.

\begin{claim}
	$\hat{\sigma}(a_n) \to 1$.
\end{claim}
\begin{proof}[Proof of Claim]
	For $n \in \N$, let $j_n \in \N$ be such that $N_{j_n} \le n < N_{j_n + 1}$.
	Then for $p > j_n$, property \eqref{eq: near 1, small index} gives
		\begin{align*}
		\int_\dual{\left| \chi(a_n) - 1 \right|~d\sigma_p(\chi)} < 2^{-(j_n - 1)}.
	\end{align*}
	
	\noindent Since $\chi \mapsto \left| \chi(a_n) - 1 \right|$ is a continuous function on $\dual$,
		\begin{align*}
		\int_\dual{\left| \chi(a_n) - 1 \right|~d\sigma(\chi)} \le 2^{-(j_n - 1)} \tendsto{n \to \infty} 0.
	\end{align*}
\end{proof}

\begin{claim} \label{claim: cts meas}
	$\sigma$ is continuous.
\end{claim}
\begin{proof}[Proof of Claim]
	Fix $q \in \N$ and let $p \ge q$.
	For $1 \le i \le 2^p$, we can write $i = l \cdot 2^q + r$ for some $l \ge 0$ and $1 \le r \le 2^q$.
	By property \eqref{eq: Cantor}, we have $\chi_i \in V_{q,r}$.
	Hence,
		\begin{align*}
		\sigma_p(V_{q,r}) = \sigma_p \left( \overline{V}_{q,r} \right)
		 = 2^{-p} \#\{ 1 \le i \le 2^p : i \equiv r \pmod{2^q} \} = 2^{-q}.
	\end{align*}
	
	\noindent By the Portmanteau theorem (see, e.g., \cite[Theorem 2.1]{billingsley}),
	it follows that $\sigma(V_{q,r}) = 2^{-q}$.
	
	Now suppose for contradiction that $\sigma$ has an atom, $\sigma(\{\chi\}) > 0$.
	Then $\chi \in V_{q,r}$ for some $1 \le r \le 2^q$.
	Hence, $\sigma(\{\chi\}) \le \sigma(V_{q,r}) = 2^{-q}$.
	Letting $q \to \infty$, we get $\sigma(\{\chi\}) = 0$, a contradiction.
	Thus, $\sigma$ is a continuous measure on $\dual$.
\end{proof}

It remains to construct the sequences $(\chi_i)_{i \in \N}$, $(N_p)_{p \ge 0}$, and $(V_{p,r})_{1 \le r \le 2^p, p \ge 0}$.
We will construct the sequences by induction on $p$,
getting $\chi_i$ for $2^{p-1} + 1 \le i \le 2^p$ at each step along with $N_p$.

We begin with $N_0 = 0$, $\chi_1 = 1$, and $V_{0,1} = \dual$.
Consider $p = 1$.
Choose $\chi_2 \in C, \chi_2 \ne 1$.
Then $\sigma_1 = \frac{\delta_1 + \delta_{\chi_2}}{2}$, and we have
\begin{align*}
	\int_\dual{\left| \chi(a_n) - 1 \right|~d\sigma_1(\chi)} = \frac{1}{2} \left| \chi_2(a_n) - 1 \right| \le 1 < 2.
\end{align*}

\noindent Now since $\chi_2 \in C$, we may choose $N_1 \in \N$ so that for $n \ge N_1$,
$\left| \chi_2(a_n) - 1 \right| < \frac{1}{2}$.
Then for $n \ge N_1$,
\begin{align*}
	\int_\dual{\left| \chi(a_n) - 1 \right|~d\sigma_1(\chi)} = \frac{1}{2} \left| \chi_2(a_n) - 1 \right| < 2^{-2}.
\end{align*}

\noindent Since $\chi_1 \ne \chi_2$, we can find disjoint neighborhoods $V_{1,1} \ni \chi_1$ and $V_{1,2} \ni \chi_2$.

Suppose we have completed step $p$ of the induction.
Now we will do sub-induction on $1 \le s \le 2^p$ to construct distinct characters
$\chi_{2^p + s} \in C$, sets $V_{p+1,s}$ and $V_{p+1,2^p+s}$, measures $\sigma_{p,s} \in \P(\dual)$,
and $N_p < N_{p,1} < \cdots < N_{p, 2^p}$ satisfying the following properties:

\begin{itemize}
	\item	For $0 \le j \le p - 1$ and $N_j \le n \le N_{j+1}$,
		
		\begin{align} \label{eq: small n}
			\int_\dual{\left| \chi(a_n) - 1 \right|~d\sigma_{p,s}(\chi)} & < 2^{-(j-1)}.
	\intertext{\item	For $n \ge N_p$,} \label{eq: medium n}
			\int_\dual{\left| \chi(a_n) - 1 \right|~d\sigma_{p,s}(\chi)} & < 2^{-(p-1)}.
	\intertext{\item	For $n \ge N_{p,s}$,} \label{eq: large n}
			\int_\dual{\left| \chi(a_n) - 1 \right|~d\sigma_{p,s}(\chi)} & < 2^{-(p+2)}.
	\intertext{\item	The sets $V_{p+1,s}$ and $V_{p+1,2^p+s}$ are disjoint subsets of $V_{p,s}$,
	and for $\ep \in \{0,1\}$}
		\label{eq: Cantor nesting}
			\chi_{\ep \cdot 2^p + s} & \in V_{p+1,\ep \cdot 2^p + s}.
		\end{align}
		
\end{itemize}

First we consider $s = 1$.
Choose $\chi_{2^p + 1} \in (C \cap V_{p,1}) \setminus \{\chi_1, \dots, \chi_{2^p}\}$.
Set
\begin{align*}
	\sigma_{p,1} & := \sigma_p + 2^{-(p+1)} \left( \delta_{\chi_{2^p+1}} - \delta_1 \right). \\
	 & = 2^{-p} \left( \frac{\delta_1 + \delta_{\chi_{2^p+1}}}{2}
	 + \sum_{i=2}^{2^p}{\delta_{\chi_i}} \right).
\end{align*}

\noindent Then for $n \ge 0$, we have
\begin{align*}
	\int_\dual{\left| \chi(a_n) - 1 \right|~d\sigma_{p,1}(\chi)}
	 \le \int_\dual{\left| \chi(a_n) - 1 \right|~d\sigma_p(\chi)}
	 + 2^{-(p+1)} \left| \chi_{2^p+1}(a_n) - 1 \right|.
\end{align*}

\noindent Taking $\chi_{2^p+1}$ sufficiently close to $1$,
the strict inequality in property \eqref{eq: near 1, small index}
allows for property \eqref{eq: small n} being satisfied for $s = 1$.
Moreover, for $n \ge N_p$, we can use the bound $|\chi_{2^p+1}(a_n) - 1| \le 2$
to see that property \eqref{eq: medium n} follows from property \eqref{eq: near 1, large index}.
Finally, since $\chi_1, \dots, \chi_{2^p+1} \in C$, we can find $N_{p,1}$ to satisfy \eqref{eq: large n}.
The points $\chi_1$ and $\chi_{2^p+1}$ are distinct, so we can find disjoint open subsets
$V_{p+1,1}, V_{p+1,2^p+1} \subseteq V_{p,1}$ such that $\chi_1 \in V_{p+1,1}$ and $\chi_{2^p+1} \in V_{p+1,2^p+1}$.
This completes the base case of the induction.

Now suppose we have chosen $\chi_{2^p+s'}$, $\sigma_{2^p+s'}$, and $N_{2^p+s'}$
satisfying \eqref{eq: small n}, \eqref{eq: medium n}, and \eqref{eq: large n} for $s' < s$.
Then we choose $\chi_{2^p+s} \in (C \cap V_{p,s}) \setminus \{\chi_1, \dots, \chi_{2^p+s-1} \}$, and define
\begin{align*}
	\sigma_{p,s} := \sigma_{p,s-1} + 2^{-(p+1)} \left( \delta_{\chi_{2^p+s}} - \delta_{\chi_s} \right).
\end{align*}

\noindent As in the $s=1$ case, we have an inequality
\begin{align*}
	\int_\dual{\left| \chi(a_n) - 1 \right|~d\sigma_{p,s}(\chi)}
	 \le & \int_\dual{\left| \chi(a_n) - 1 \right|~d\sigma_{p,s-1}(\chi)} \\
	 & + 2^{-(p+1)} \left| \chi_{2^p+s}(a_n) - \chi_s(a_n) \right|,
\end{align*}

\noindent so choosing $\chi_{2^p+s}$ sufficiently close to $\chi_s$ implies that
\eqref{eq: small n} holds by induction.
We can also choose $\chi_{2^p+s}$ sufficiently close to $\chi_s$ so that \eqref{eq: medium n} holds
for the finitely many values $N_p \le n \le N_{p,s-1}$.
Then for $n \ge N_{p,s-1}$, property \eqref{eq: large n} for $\sigma_{p,s-1}$
implies \eqref{eq: medium n} for $\sigma_{p,s}$.
As in the $s=1$ case, since all of the characters are taken from $C$,
property \eqref{eq: large n} holds for all sufficiently large $N_{p,s}$.
Finally, let $V_{p+1,s}$ and $V_{p+1,2^p+s}$ be disjoint open subsets of $V_{p,s}$
such that $\chi_s \in V_{p+1,s}$ and $\chi_{2^p+s} \in V_{p+1,2^p+s}$.

Now that we have completed induction on $s$, we return to the induction on $p$.
Define $\sigma_{p+1} := \sigma_{p,2^p}$ and $N_{p+1} := N_{p,2^p}$.
Property \eqref{eq: small n} implies property \eqref{eq: near 1, small index} for $0 \le j \le p-1$.
The $j = p$ case of property \eqref{eq: near 1, small index} follows from \eqref{eq: medium n},
so \eqref{eq: near 1, small index} holds.
Property \eqref{eq: near 1, large index} follows immediately from \eqref{eq: large n}.
Thus, it remains only to check property \eqref{eq: Cantor}, which we do with a simple induction argument.

Assume that \eqref{eq: Cantor} holds for $p$.
We want to show that it holds for $p+1$.
Let $0 \le q \le p$, $0 \le l \le 2^{p+1-q} - 1$, and $1 \le r \le 2^q$.
Write $l = l' + \ep \cdot 2^{p-q}$ with $\ep \in \{0,1\}$ and $0 \le l' \le 2^{p-q} - 1$
so that $l \cdot 2^q + r = l' \cdot 2^q + r + \ep \cdot 2^p$.
Let $s := l' \cdot 2^q + r$.
Then $1 \le s \le (2^{p-q} - 1) 2^q + 2^q = 2^p$ and $l \cdot 2^q + r = \ep \cdot 2^p + s$.
Thus, by property \eqref{eq: Cantor nesting}, we have
\begin{align*}
	\chi_{l \cdot 2^q + r} \in V_{l \cdot 2^q + r} = V_{p+1,\ep \cdot 2^p + s} \subseteq V_{p,s}.
\end{align*}

\noindent By the induction hypothesis, we have $V_{p,s} = V_{p,l' \cdot 2^q + r} \subseteq V_{q,r}$.
Thus, property \eqref{eq: Cantor} holds, completing the induction. \qed


\subsection{Consequences of Theorem \ref{thm: rigidWM}}

As a first consequence of Theorem \ref{thm: rigidWM}, we prove Theorem \ref{thm: free}:

\free*

We recall some basic facts about Pontryagin duality that will be used in the proof.
The \emph{annihilator} of a set $E \subseteq \Gamma$ is the closed subgroup
\begin{align*}
	E^{\perp} := \left\{ \chi \in \dual : \chi(g) = 1~\text{for all}~g \in E \right\} \le \dual.
\end{align*}

\noindent Similarly, for a set $M \subseteq \dual$, the \emph{annihilator} is the subgroup
\begin{align*}
	M^{\perp} := \left\{ g \in \Gamma : \chi(g) = 1~\text{for all}~\chi \in M \right\} \le \Gamma.
\end{align*}

\noindent In general, the annihilator of the annihilator of a set is the closed subgroup generated by that set.
Moreover, given closed subgroups $\Delta \le \Gamma$ and $K \le \dual$ with $\Delta^{\perp} = K$,
we have isomorphisms $K \simeq \hat{\Gamma/\Delta}$ and $\dual/K \simeq \hat{\Delta}$
(see, e.g., \cite[Section 2.1]{rudin}).

As a first step towards proving Theorem \ref{thm: free}, we need a criterion for checking that a Gaussian system is free.
For a measure $\sigma \in \P(\dual)$, let $G(\sigma)$ be the closed subgroup generated by $\supp{\sigma}$.

\begin{lemma} \label{lem: free WM}
	Let $\sigma \in \P(\dual)$ be a probability measure.
	Then
		\begin{align*}
		G(\sigma) = \{g \in \Gamma : \hat{\sigma}(g) = 1\}^{\perp}.
	\end{align*}
	
	\noindent In particular, the Gaussian system associated to a measure $\sigma \in \P(\dual)$ is free
	if and only if $G(\sigma) = \dual$.
\end{lemma}
\begin{proof}
	Since $\sigma$ is a probability measure and
		\begin{align*}
		\hat{\sigma}(g) = \int_{\dual}{\chi(g)~d\sigma(\chi)},
	\end{align*}
	
	\noindent we have $\hat{\sigma}(g) = 1$ if and only if $\chi(g) = 1$ for $\sigma$-a.e. $\chi \in \dual$.
	Noting that, for any $g \in \Gamma$, the set $\{\chi \in \dual : \chi(g) = 1\}$ is a closed subgroup,
	we have that $\hat{\sigma}(g) = 1$ if and only if $\chi(g) = 1$ for every $\chi \in G(\sigma)$.
	That is, $\{g \in \Gamma : \hat{\sigma}(g) = 1\} = G(\sigma)^{\perp}$, so we are done.
\end{proof}

Similar considerations for discrete spectrum systems produce the following:

\begin{lemma} \label{lem: free DS}
	Suppose $\left( X, \B, \mu, (T_g)_{g \in \Gamma} \right)$ is an ergodic measure-preserving system
	with discrete spectrum that is rigid along a sequence $(a_n)_{n \in \N}$.
	Then
		\begin{align*}
		C := \{ \chi \in \dual : \chi(a_n) \to 1 \}
	\end{align*}
	
	\noindent is dense in the subgroup $\{g \in \Gamma : T_g = \id\}^{\perp} \subseteq \dual$.
	In particular, if the system is free, then $C$ is dense in $\dual$.
\end{lemma}
\begin{proof}
	Let $\Gamma_0 := \{g \in \Gamma : T_g = \id\}$.
	If $\chi \in \dual$ is an eigenvalue for the system, then $\chi(a_n) \to 1$
	(see the proof of Lemma \ref{lem: rigid characters}).
	So, it suffices to check that the group of eigenvalues
		\begin{align*}
		\Lambda := \left\{ \chi \in \dual : \text{there exists}~f \in L^2(\mu)
		 ~\text{such that}~T_gf = \chi(g)f~\text{for every}~g \in \Gamma \right\}
	\end{align*}
	
	\noindent is dense in $\Gamma_0^{\perp}$.
	Equivalently, we want to show $\Lambda^{\perp} = \Gamma_0$.
	
	Let $g \in \Lambda^{\perp}$.
	By definition, $\chi(g) = 1$ for every $\chi \in \Lambda$.
	Since $L^2(\mu)$ is spanned by eigenfunctions, it follows that $T_g = \id$.
	That is, $g \in \Gamma_0$.
	
	Conversely, suppose $g \in \Gamma_0$.
	Let $\chi \in \Lambda$, and let $f \in L^2(\mu)$ such that $T_hf = \chi(h)f$ for every $h \in \Gamma$.
	Since $T_g = \id$, we have $T_gf = f$, so $\chi(g) = 1$.
	Thus, $g \in \Lambda^{\perp}$.
\end{proof}

Now we can prove Theorem \ref{thm: free}:

\begin{proof}[Proof of Theorem \ref{thm: free}]
	Suppose $(a_n)_{n \in \N}$ is a rigidity sequence for a free ergodic\break measure-preserving system
	$\left( X, \B, \mu, (T_g)_{g \in \Gamma} \right)$.
	
	Rather than producing a free weakly mixing system directly,
	it suffices by Lemmas \ref{lem: rigid characters} and \ref{lem: free WM} to construct
	a probability measure $\sigma \in \P(\dual)$ satisfying three conditions:
	
	\begin{enumerate}[(1)]
		\item	$\sigma$ is continuous;
		\item	$\hat{\sigma}(a_n) \to 1$;
		\item	$G(\sigma) = \dual$.
	\end{enumerate}
	
	The method for constructing such a measure depends on whether or not
	the system $\left( X, \B, \mu, (T_g)_{g \in \Gamma} \right)$ has discrete spectrum,
	so we break the proof into two cases from here.
	
	\vspace*{4pt}\noindent
	\textbf{Case 1.} $\left( X, \B, \mu, (T_g)_{g \in \Gamma} \right)$ does not have discrete spectrum.
	
	Let $\nu \in \P(\dual)$ be the maximal spectral type of $\left( X, \B, \mu, (T_g)_{g \in \Gamma} \right)$.
	We claim $\{g \in \Gamma : \hat{\nu}(g) = 1\} = \{0\}$.
	Indeed, suppose $\hat{\nu}(g) = 1$.
	Then $\chi(g) = 1$ for $\nu$-a.e. $\chi \in \dual$.
	Hence, for any $f \in L^2(\mu)$, since $\sigma_f \ll \nu$, we have $\chi(g) = 1$ for $\sigma_f$-a.e. $\chi \in \dual$.
	Therefore, $\hat{\sigma}_f(g) = \|f\|_2^2$, so $T_gf = f$.
	Since the system is free, this can only happen for $g = 0$.
	
	Decompose $\nu = t \nu_a + (1-t) \nu_c$, where $t \in (0,1)$,
	$\nu_a \in \P(\dual)$ is purely atomic and $\nu_c \in \P(\dual)$ is continuous.
	We claim that the convolution $\sigma := \nu_a * \nu_c$ satisfies the desired conditions.

	(1) The measure $\nu_c$ is continuous, so $\sigma$ is also continuous (see \cite[Theorem 19.16]{hr1}).
	
	(2) By the convolution theorem (see \cite[Theorem 1.3.3(b)]{rudin}),
	$\hat{\sigma}(g) = \hat{\nu}_a(g) \cdot \hat{\nu}_c(g)$ for every $g \in \Gamma$.
	Since $(a_n)_{n \in \N}$ is a rigidity sequence for $\left( X, \B, \mu, (T_g)_{g \in \Gamma} \right)$,
	we have $\hat{\nu}(a_n) \to 1$.
	Thus, $\hat{\nu}_a(a_n) \to 1$ and $\hat{\nu}_c(a_n) \to 1$, so $\hat{\sigma}(a_n) \to 1$ as desired.
	
	(3) By Lemma \ref{lem: free WM}, it suffices to show that $\hat{\sigma}(g) \ne 1$ for $g \in \Gamma \setminus \{0\}$.
	Let $g \in \Gamma \setminus \{0\}$.
	We know $\hat{\nu}(g) \ne 1$.
	Hence, either $\hat{\nu}_a(g) \ne 1$ or $\hat{\nu}_c(g) \ne 1$.
	In either case, taking the product, we have $\hat{\sigma}(g) = \hat{\nu}_a(g) \cdot \hat{\nu}_c(g) \ne 1$.

		\vspace*{4pt}\noindent\textbf{Case 2.} $\left( X, \B, \mu, (T_g)_{g \in \Gamma} \right)$ has discrete spectrum.
	
	By Lemma \ref{lem: free DS}, $C := \{ \chi \in \dual : \chi(a_n) \to 1 \}$ is dense in $\dual$.
	Let $C_0 \subseteq C$ be a countable dense subset
	(for example, we could take $C_0$ to be the group of eigenvalues, but such a concrete description is not needed).
	For each $\chi \in C_0$, let $\sigma_{\chi} \in \P_c(\dual)$ be a measure constructed as in the proof
	of Theorem \ref{thm: DS implies WM} with $\chi_2 = \chi$.
	Let $\alpha : C_0 \to (0,1)$ so that $\sum_{\chi \in C_0}{\alpha(\chi)} = 1$,
	and set $\sigma := \sum_{\chi \in C_0}{\alpha(\chi) \sigma_{\chi}}$.
	We claim that $\sigma$ has the desired properties.
	
	(1) Each $\sigma_{\chi}$ is continuous, so their weighted sum $\sigma$ is also continuous.
	
	(2) By construction, $\hat{\sigma}_{\chi}(a_n) \to 1$ for every $\chi \in C_0$.
	Hence, by the dominated convergence theorem,
		\begin{align*}
		\hat{\sigma}(a_n) = \sum_{\chi \in C_0}{\alpha(\chi) \hat{\sigma}_{\chi}(a_n)}
		 \to \sum_{\chi \in C_0}{\alpha(\chi)} = 1.
	\end{align*}
	
	(3) Since $\alpha(\chi) > 0$ for each $\chi \in C_0$, $\supp{\sigma}$ contains $\supp{\sigma_{\chi}}$.
	We want to check that $\chi \in \supp{\sigma_{\chi}}$ for each $\chi \in C_0$.
	Let $(U_p)_{p \ge 0}$ be a neighborhood basis at $0$ in $\dual$.
	Without loss of generality, we may assume $U_0 \supseteq U_1 \supseteq \cdots$.
	(For example, we could take $U_p$ to be a ball of radius $2^{-p}$ around 0 in some metric on $\dual$.)
	Now, in the proof of Theorem \ref{thm: DS implies WM}, we may impose the additional assumption
	that $V_{p,r}$ is a subset of $\chi_r + U_p$.
	(This can be done by replacing $V_{p,r}$ with the set $V_{p,r} \cap (\chi_r + U_p)$
	at each step of the induction process.)
	Now let $U$ be a neighborhood of $\chi$ in $\dual$.
	Then for some $p \ge 0$, we have $\chi + U_p \subseteq U$, so $V_{p,2} \subseteq U$.
	Therefore, $\sigma_{\chi}(U) \ge \sigma_{\chi}(V_{p,2}) = 2^{-p} > 0$.
	This show that $\chi \in \supp{\sigma_{\chi}}$, so $\chi \in \supp{\sigma}$ for every $\chi \in C_0$.
	But $C_0$ is dense in $\dual$, so $\supp{\sigma} = \dual$.
\end{proof}

We now deduce topological corollaries from Theorem \ref{thm: rigidWM},
including an extension of \cite[Theorem 3.8]{bgm}.
First, we need two definitions from \cite{bdmw}.

\begin{definition}
	A (Hausdorff) group topology $\tau$ on $\Gamma$ is \emph{pre-compact}
	if for every $\es \ne U \in \tau$, there is a finite set $F \subseteq \Gamma$ so that $U + F = \Gamma$.
\end{definition}

The term \emph{pre-compact} alludes to the existence of a \emph{compactification} with a compatible topology.
A topological space $X$ is a \emph{compactification} of another topological space $Y$
if $X$ is compact and there is a continuous injection $\iota : Y \to X$
such that $\iota(Y)$ is dense in $X$.
It is not hard to show that a group topology $\tau$ is pre-compact if and only if it is the subspace topology induced by some compactification of $\Gamma$.
In what follows, we assume that all topologies are Hausdorff.

\begin{definition}
	A sequence $(a_n)_{n \in \N}$ in $\Gamma$ is \emph{totally bounded} (a \emph{TB-sequence} for short)
	if there is a pre-compact group topology $\tau$ on $\Gamma$ such that $a_n \to 0$ in $(\Gamma, \tau)$.
\end{definition}

In \cite{bdmw}, it is shown that a sequence $(n_k)_{k \in \N}$ in $\Z$ is a TB-sequence
if and only if $\lambda^{n_k} \to 1$ for infinitely many $\lambda \in \T$.
By Lemma \ref{lem: rigid characters}, this means that TB-sequences in $\Z$ are rigidity sequences
for ergodic systems with discrete spectrum.
Badea, Grivaux, and Matheron use this observation to conclude that TB-sequences are rigidity sequences
(see \cite[Theorem 3.8]{bgm}).

In general groups, there may be rigidity sequences for discrete spectrum systems
that are not TB-sequences (see Remark \ref{rem: TB-seq} below).
However, every TB-sequence is still rigid for some system with discrete spectrum:

\begin{corollary} \label{cor: TB implies rig}
	TB-sequences are rigidity sequences.
\end{corollary}
\begin{proof}
	Let $\tau$ be a pre-compact group topology on $\Gamma$ such that $a_n \to 0$.
	Let $X$ be the corresponding compactification of $\Gamma$.
	Since $\Gamma$ is dense in $X$, the action of $\Gamma$ on itself by translation
	extends to a minimal action $T: \Gamma \acts X$.
	The space $X$ has a group structure, and $T$ acts by group translations.
	Any minimal such system is uniquely ergodic.
	In particular, $(T_g)_{g \in \Gamma}$ is ergodic and preserves the Haar measure $\mu$ on $X$.
	Moreover, since $a_n \to 0$ in $(\Gamma, \tau)$, we immediately have $T_{a_n} \to \id$,
	so $(a_n)_{n \in \N}$ is a rigidity sequence by Theorem \ref{thm: rigidWM}.
\end{proof}

\begin{remark}
	The system constructed in the preceding proof has discrete spectrum
	(see Proposition \ref{prop: DS equiv}(ii)).
	Moreover, since the topology $\tau$ is Hausdorff, $\Gamma$ acts freely on $X$.
	This is not needed to conclude that TB-sequences are rigidity sequences,
	but it does lead towards a dynamical characterization of TB-sequences
	(see Proposition \ref{prop: free rig implies null} and Remark \ref{rem: TB-seq} below).
\end{remark}

More generally, following Ruzsa \cite{ruzsa}, we say that a sequence $(a_n)_{n \in \N}$ is \emph{nullpotent}
if there is some Hausdorff group topology $\tau$ on $\Gamma$ such that $a_n \to 0$ in $(\Gamma, \tau)$.
Such sequences are also called \emph{T-sequences} by Protasov and Zelenyuk \cite{pz},
but we stick to Ruzsa's terminology to be consistent with \cite{bgm}.
Badea, Grivaux, and Matheron showed that every rigidity sequence in $\Z$ is nullpotent (see \cite[Corollary 3.6]{bgm}).
We extend this to our setting under the assumption that the measure-preserving action is free
(as discussed in the introduction, this assumption is extraneous in $\Z$).

\begin{proposition} \label{prop: free rig implies null}
	Suppose $(a_n)_{n \in \N}$ is rigid for a free ergodic action $\Gamma \acts (X, \B, \mu)$.
	Then $(a_n)_{n \in \N}$ is nullpotent.
\end{proposition}
\begin{proof}
	Let $f \in L^2(\mu)$ so that $T_gf \ne f$ for every $g \in \Gamma \setminus \{0\}$.
	Such an $f$ exists since the action is free (all that is required is that $f$ is injective).
	Then we can define a metric $d$ on $\Gamma$ by $d(g,h) := \left\| T_gf - T_hf \right\|$.
	Note that $d$ is translation-invariant, so it induces a group topology $\tau$ on $\Gamma$.
	Moreover, $d(a_n,0) = \|T_{a_n}f - f\| \to 0$.
	Thus, $(a_n)_{n \in \N}$ is nullpotent.
\end{proof}

\begin{remark} \label{rem: TB-seq}
	The same proof shows that if $\Gamma \acts (X, \B, \mu)$ is a free ergodic action
	with discrete spectrum for which $(a_n)_{n \in \N}$ is rigid, then $(a_n)_{n \in \N}$ is a TB-sequence,
	providing a partial converse to Corollary \ref{cor: TB implies rig}.
	This yields an alternative (dynamical) proof to the characterization of TB-sequences
	given in \cite[Proposition 3.2]{dmt}:
	$(a_n)_{n \in \N}$ is a TB-sequence if and only if $\{\chi \in \dual : \chi(a_n) \to 1\}$ is dense in $\dual$.
\end{remark}

Combining the above observations, we get the following topological
version of Theorems \ref{thm: rigidWM} and \ref{thm: free}:

\begin{corollary} \label{cor: top DS implies WM}
	Let $(a_n)_{n \in \N}$ be a sequence in $\Gamma$.
	If there is a (Hausdorff) group topology $\tau$ on $\Gamma$
	such that $a_n \to 0$ in $(\Gamma, \tau)$,
	then there is a non-pre-compact (Hausdorff) group topology $\tau'$ on $\Gamma$
	such that $a_n \to 0$ in $(\Gamma, \tau')$.
\end{corollary}


\subsection{Analyzing the assumption of Theorem \ref{thm: rigidWM}}

It is worth examining where we have used various properties of $\Gamma$.
The translation into Fourier analysis relies on $\Gamma$ being abelian.
For non-abelian groups, some more intricate analysis of the relationship
between finite and infinite-dimensional representations may have some utility,
but it is not at all clear how to extend the argument used here to the non-abelian setting.

Discreteness of the group $\Gamma$ cannot be dropped
from the assumptions of Theorem \ref{thm: rigidWM}.
As an example, Theorem \ref{thm: rigidWM} fails for $\Q$
when equipped with the subspace topology from $\R$.
Consider the $\Q$-action on $\T$ given by $T_q(x) = x + q \pmod{1}$ for $q \in \Q$.
This action is clearly continuous when $\Q$ is endowed with the subspace topology.
Moreover, $\Z$ is rigid for this action.
In fact, $T_n = \id$ for every $n \in \Z$.
By Theorem \ref{thm: rigidWM}, it follows that there is a weakly mixing action of $\Q$
(as a discrete group) for which $\Z$ is rigid.
The following propositions shows that any such action must be
discontinuous in the subspace topology on $\Q$:

\begin{proposition}
	Suppose $S \subseteq \Q$ has bounded gaps.
	That is, there is an $L > 0$ so that $S \cap [t, t+L] \ne \es$ for every $t \in \Q$.
	Then no enumeration of $S$ can be a rigidity sequence in $\Q$
	with the subspace topology.
\end{proposition}
\begin{proof}
	Suppose $\left( X, \B, \mu, (T_q)_{q \in \Q} \right)$ is a measure-preserving system
	with $(T_q)_{q \in \Q}$ acting continuously in the subspace topology,
	and suppose that (some enumeration of) $S$ is rigid for this action.
	We will show that $(T_q)_{q \in \Q}$ has discrete spectrum.
	
	Fix $f \in L^2(\mu)$.
	We want to show that the orbit $\{T_qf : q \in \Q\}$ is pre-compact (totally bounded) in $L^2(\mu)$.
	Let $\ep > 0$.
	Since $S$ is rigid, the set
		\begin{align*}
		P_{\ep/2}(f) := \left\{ \tau \in \Q : \|T_{\tau}f - f\| < \frac{\ep}{2} \right\}
	\end{align*}
	
	\noindent contains all but finitely many elements of $S$, so it still has bounded gaps.
	Let $L = L(\ep) > 0$ so that $P_{\ep/2}(f) \cap [t, t +L] \ne \es$ for every $t \in \Q$.
	Since the action is continuous, let $\delta > 0$
	so that if $|q| < \delta$, then $\|T_qf - f\| < \frac{\ep}{2}$.
	Let $F \subseteq \Q \cap [0,L]$ be finite and $\delta$-dense.
	
	We claim that $\{T_tf : t \in F\}$ is $\ep$-dense in $\{T_qf : q \in \Q\}$.
	Let $q \in \Q$.
	Choose $s \in P_{\ep/2}(f) \cap [q-L, q]$.
	Then $q-s \in [0,L]$, so there is a $t \in F$ with $|q-s-t| < \delta$.
	Thus,
		\begin{align*}
		\|T_qf - T_tf\| & \le \|T_qf - T_{q-s}f\| + \|T_{q-s}f - T_tf\| \\
		 & = \|T_sf - f\| + \|T_{q-s-t}f - f\| < \frac{\ep}{2} + \frac{\ep}{2} = \ep.
	\end{align*}
	
	\noindent This proves that $\{T_tf : t \in F\}$ is $\ep$-dense, as claimed.
	Therefore, $\{T_qf : q \in \Q\}$ is totally bounded.
	Thus, $(T_q)_{q \in \Q}$ has discrete spectrum
	(condition (iii) of Proposition \ref{prop: DS equiv} is satisfied).
\end{proof}

Finally, although countability of $\Gamma$ does not come into play in Theorem \ref{thm: rigidWM}
(the Cantor-like construction uses only that $\dual$ is a compact Hausdorff space),
it is crucial to our proof of Theorem \ref{thm: free}.
Indeed, in order to construct a continuous measure $\sigma_{\chi}$ for which a pre-specified $\chi$ is in the support,
we needed the open sets $V_{p,2}$ to shrink to $\chi$ as $p \to \infty$
(see (3) in Case 2 of the proof of Theorem \ref{thm: free} above).
By \cite[Theorem 8.6]{hr1}, if $\{\chi\}$ can be written as the intersection of countably many open sets in $\dual$,
then $\dual$ is metrizable, and this happens if and only if $\Gamma$ is countable.


\section{Examples of rigidity sequences} \label{sec: examples}

Now that we have established Theorem \ref{thm: rigidWM},
we give several examples of nontrivial rigidity sequences for weakly mixing actions.


\subsection{Pisot--Vijayaraghavan numbers in $\R$}

A number $\alpha \in \R$ is a \emph{Pisot--\break Vijayaraghavan (PV) number}
if $\alpha > 1$ is an algebraic integer all of whose Galois conjugates have absolute value less than 1.
PV numbers have the property that $\|\alpha^n\| \to 0$,
where $\|\cdot\|$ is the distance to the nearest integer.
Examples of PV numbers include the plastic number (the real root of $x^3-x-1$) and the golden ratio.

\begin{proposition} \label{prop: PV rigid}
	If $\alpha$ is a PV number, then $\round{\alpha^n}$ is a rigidity sequence in $\Z$,
	where $\round{\cdot}$ is the nearest integer function.
\end{proposition}
\begin{proof}
	Since any irrational rotation of the circle is ergodic,
	it suffices to check that $\left\| \round{\alpha^n} \beta \right\| \to 0$
	for some irrational $\beta \in \R \setminus \Q$.
	We claim that $\beta = \alpha$ works.
	Indeed,
		\begin{align*}
		\left\| \round{\alpha^n} \alpha \right\|
		 = \left\| \alpha^{n+1} \pm \| \alpha^n \| \cdot \alpha \right\|
		 \le \left\| \alpha^{n+1} \right\| + \left\| \| \alpha^n \| \cdot \alpha \right\|
		 \to 0.
	\end{align*}
	\end{proof}

\begin{remark}
	The key to showing that $\round{\alpha^n}$ is a rigidity sequence is the existence of $\beta$
	such that $\left\| \alpha^n\beta \right\| \to 0$.
	It is known that for algebraic $\alpha$, this condition implies that $\alpha$ is a PV number.
	However, it is an open question whether there is a transcendental $\alpha$ such that
	$\left\| \alpha^n\beta \right\| \to 0$ for some $\beta \ne 0$.
\end{remark}


\subsection{PV numbers in finite characteristic}

We can extend the previous example to a parallel situation in $\F_q[t]$.
PV numbers have a natural analogue in this setting, described in detail in \cite{bd}.
We first introduce some suggestive notation.
Let $\scrZ = \F_q[t]$, $\scrQ = \F_q(t)$, $\scrR = \F_q((t^{-1}))$, and $\scrT = \scrR/\scrZ$.
Define $|\cdot|$ on $\scrR$ by
\begin{align*}
	\left| \sum_{n=-\infty}^N{c_nt^n} \right| = q^N,
\end{align*}

\noindent where we assume $c_N \ne 0$.
As with the real line, $\scrR$ is the completion of $\scrQ$ with respect to $|\cdot|$.
Here, $\scrT$ can be seen as the subgroup $t^{-1}\F_q[[t^{-1}]]$ of $\scrR$ so that $\scrR = \scrZ \oplus \scrT$.
For $x \in \scrR$, let $\floor{x} \in \scrZ$ and $\fpart{x} \in \scrT$
so that $x = \floor{x} + \fpart{x}$.
That is, if $x = \sum_{n=-\infty}^{N}{c_nt^n}$, then
\begin{align*}
	\floor{x} & := \sum_{n=0}^{N}{c_nt^n}, \\
	\fpart{x} & := \sum_{n=-\infty}^{-1}{c_nt^n}.
\end{align*}

\noindent We let $\|\cdot\|$ denote the distance to the nearest ``integer'' $\|x\| = |\{x\}|$.

The absolute value $|\cdot|$ on $\scrR$ extends to an absolute value on the algebraic closure $\scrA$.
Suppose $\alpha \in \scrA$ is of degree $d$ over $\scrR$,
and let $N = N_{\scrR(\alpha)/\scrR}$ be the norm on $\scrR(\alpha)/\scrR$.
Then we define
\begin{align*}
	|\alpha| := |N(\alpha)|^{1/d}.
\end{align*}

\noindent This allows us to define PV numbers in the same way as in $\R$.
Namely, $\alpha \in \scrR \setminus \F_q$ is a \emph{Pisot--Vijayaraghavan element} if
$\alpha$ is integral over $\scrZ$ and all of its conjugates have absolute value less than 1 in $\scrA$.
With this definition, Bateman and Duquette show that $\scrR$ contains PV elements
of every (algebraic) degree over $\scrZ$ (see \cite[Theorem 1.1]{bd}).
They are also able to show the following strengthening of the characterization of PV numbers in $\R$:

\begin{theorem}[\cite{bd}, Theorems 2.1 and 3.4] \label{thm: PV fin char}
	Suppose $\alpha \in \scrR$ is a PV element of degree $d$ over $\scrZ$.
	Let $T_{\scrQ(\alpha)/\scrQ}$ denote the trace on the extension $\scrQ(\alpha)/\scrQ$.
	Suppose $\beta \in \scrQ(\alpha)$ such that for some $N \in \N$,
		\begin{align*}
		T_{\scrQ(\alpha)/\scrQ}\left( \alpha^{N+i}\beta \right) \in \scrZ
	\end{align*}
	
	\noindent for $0 \le i \le d - 1$.
	Then $\|\alpha^n \beta\| \to 0$.
	
	Conversely, if $\alpha \in \scrR$ with $|\alpha| > 1$,
	and $\beta \in \scrR$ is nonzero element such that $\|\alpha^n\beta\| \to 0$,
	then $\alpha$ is a PV element, and $\beta$ is an element of $\scrQ(\alpha)$
	such that for some $N \in \N$,
		\begin{align*}
		T_{\scrQ(\alpha)/\scrQ}\left( \alpha^{N+i}\beta \right) \in \scrZ
	\end{align*}
	
	\noindent for $0 \le i \le d - 1$.
\end{theorem}

In particular, if $\alpha \in \scrR$ is a PV element, then $\|\alpha^n\| \to 0$.
Repeating the argument from the proof of Proposition \ref{prop: PV rigid}, we have the following.

\begin{proposition} \label{prop: PV rig fin char}
	If $\alpha \in \scrR$ is a PV element, then $\floor{\alpha^n}$ is a rigidity sequence in $\scrZ$.
\end{proposition}

\begin{remark}
	The conclusion of Proposition \ref{prop: PV rig fin char} can be strengthened
	to say that $\floor{\alpha^n}$ is a rigidity sequence for a free weakly mixing action of $\scrZ$
	with the help of Theorem \ref{thm: free}.
	Indeed, the ``irrational rotation'' $T_g x := \fpart{x + g \alpha}$ is free and ergodic.
\end{remark}

Motivated by these examples involving PV numbers,
we ask whether a converse exists to Proposition \ref{prop: PV rigid} or \ref{prop: PV rig fin char}.
Namely, we pose the following question.

\begin{question}
	For which $\alpha \in \R$ is $\round{\alpha^n}$ a rigidity sequence in $\Z$?
	Similarly, for which $\alpha \in \F_q((t^{-1}))$ is $\floor{\alpha^n}$ a rigidity sequence in $\F_q[t]$?
	In particular, we are interested in whether this class of numbers
	agrees with the class of PV numbers or if it is strictly larger.
\end{question}


\subsection{Continued fractions}

Bergelson, del Junco, Lema\'{n}czyk, and Rosenblatt\break showed that in $\Z$,
if $\alpha \in \R$ is irrational, then its sequence of continued fraction denominators
is rigid for a weakly mixing transformation (see \cite[Corollary 3.51]{bdjlr}).
We show that the analogous result holds in $\F_q[t]$.
For an introduction to continued fractions in finite characteristic, see \cite{s00}.
We give the basics of the construction here.

Define a tranformation $T : \scrT \to \scrT$ by
\begin{align*}
	Tx = \begin{cases}
		\fpart{\frac{1}{x}}, & x \ne 0; \\
		0, & x = 0.
	\end{cases}
\end{align*}

\noindent For $\alpha \in \scrR$, we can then define $a_0 := \floor{\alpha}$
and $a_n := \floor{T^n\fpart{\alpha}}$ for $n \ge 1$
to obtain a continued fraction expression for $\alpha$:
\begin{align*}
	\alpha = a_0 + \cfrac{1}{a_1 + \cfrac{1}{a_2 + \cfrac{1}{a_3 + \frac{1}{\ddots}}}}.
\end{align*}

\noindent We will also use the notation $\alpha = [a_0; a_1, a_2, a_3, \dots]$.

Let $(q_n)_{n \in \N}$ be the sequence of denominators
for the convergents $\frac{p_n}{q_n} = [a_0; a_1, \dots,\break a_n]$.
These can be computed recursively by the formula
\begin{align*}
	q_0 & = 1, \\
	q_1 & = a_1, \\
	q_n & = a_n q_{n-1} + q_{n-2}.
\end{align*}

Continued fractions satisfy the inequality (see \cite{s00})
\begin{align*}
	\left| \alpha - \frac{p_n}{q_n} \right| < \frac{1}{|q_n|^2}.
\end{align*}

\noindent Thus, $\|q_n\alpha\| < |q_n|^{-1} \to 0$.

Irrational numbers have many of the familiar properties from characteristic zero
(see \cite[Theorem 3.1]{bl} for details):

\begin{lemma} \label{lem: irrational}
	The following are equivalent for $\alpha \in \scrR$:
	\begin{enumerate}[(i)]
		\item	$\alpha$ is irrational, i.e. $\alpha \notin \scrQ$;
		\item	$\alpha$ has an infinite continued fraction expansion;
		\item	The orbit $(n\alpha)_{n \in \scrZ}$ of $\alpha$ in $\scrT$ is infinite;
		\item $(n\alpha)_{n \in \scrZ}$ is dense in $\scrT$;
		\item $(n\alpha)_{n \in \scrZ}$ is well-distributed in $\scrT$ with respect to Haar measure.
	\end{enumerate}
\end{lemma}

Suppose $\alpha \in \scrR$ is irrational.
Then the $\scrZ$-action on $\scrT$ given by $T_nx := x + n \alpha \pmod{\scrZ}$
is ergodic (with respect to Haar measure) as a consequence of Lemma \ref{lem: irrational}.
Moreover, $\| T_{q_n}x - x \| = \| q_n \alpha\| \to 0$ for all $x \in \scrT$,
so $(q_n)_{n \in \N}$ is a rigidity sequence for this action.
Applying Theorem \ref{thm: free}, we have proven the following (c.f. \cite[Corollary 3.51]{bdjlr}):

\begin{theorem}
	Let $\alpha \in \scrR \setminus \scrQ$.
	Let $(q_n)_{n \in \N}$ be the sequence of denominators in $\scrZ$.
	Then there is a free weakly mixing $\scrZ$-action $(T_n)_{n \in \scrZ}$
	such that $T_{q_n} \to \id$.
\end{theorem}
	
In particular, the sequence of \emph{Fibonacci polynomials} reduced modulo $p$ is a rigidity sequence in $\F_p[t]$.
This sequence is given by the recurrence relation
\begin{align*}
	F_0(t) & = 1, \\
	F_1(t) & = t, \\
	F_n(t) & = t F_{n-1}(t) + F_{n-2}(t)
\end{align*}

\noindent corresponding to the element $\alpha = [t; t,t,\dots] \in \F_p((t^{-1}))$.
The first several polynomials are given by
\begin{align*}
	F_0(t) & = 1, \\
	F_1(t) & = t, \\
	F_2(t) & = t^2 + 1, \\
	F_3(t) & = t^3 + 2t, \\
	F_4(t) & = t^4 + 3t^2 + 1, \\
	F_5(t) & = t^5 + 4t^3 + 3t, \\
	F_6(t) & = t^6 + 5t^4 + 6t^2 + 1, \\
	F_7(t) & = t^7 + 6t^5 + 10t^3 + 4t, \\
	F_8(t) & = t^8 + 7t^6 + 15t^4 + 10t^2 + 1.
\end{align*}


\subsection{More examples in $\bigoplus{\F_p}$}

The next two examples seem to have no clear analogues in $\Z$.
Since we are no longer dealing with the multiplicative structure of $\F_q[t]$,
we will restrict for ease of notation to considering $\bigoplus{\F_p}$, where $p$ is prime.
We first need to introduce some notation.
For $x \in \scrT$ and $a \in \scrZ$, we denote by $\innprod{x}{a}$
the quantity $e(ax)$, where $e \left( \sum_{n=-\infty}^N{c_nt^n} \right) = e^{2\pi ic_{-1}/p}$.
This pairing demonstrates the duality between $\scrZ$ and $\scrT$:
for each $x \in \scrT$, the map $a \mapsto \innprod{x}{a}$ defines a character on $\scrZ$,
and for each $a \in \scrZ$, the map $x \mapsto \innprod{x}{a}$ defines a character on $\scrT$.
Handling characters concretely this way allows for easy computations in the following examples.

\begin{example}
	Consider the sequence $a_n(t) := \sum_{k=0}^{np-1}{t^k} \in \scrZ$.
	Let
		\begin{align*}
		C_0 := \left\{ \sum_{n=1}^{\infty}{c_n t^{-n}}\in \scrT :
		 c_{jp+1} = c_{jp+2} = \cdots = c_{(j+1)p}~\text{for}~j \ge 0 \right\}.
	\end{align*}
	
	\noindent Note that $C_0$ is infinite (it has a natural bijection with $\scrT$).
	Moreover, for $x \in C_0$ and $n \in \N$, $\innprod{x}{a_n} = 1$.
	Hence, $(a_n)_{n \in \N}$ is a rigidity sequence.
\end{example}

\begin{example}
	Let $I \subseteq \N \cup \{0\}$ be an infinite set with infinite complement.
	For $\es \ne F \subseteq I$ finite and $(c_i)_{i \in F} \in \F_p^F$ not all zero,
	let $a_{F; (c_i)}(t) = \sum_{i \in F}{c_i t^i}$.
	Set
		\begin{align*}
		C_0 := \left\{ \sum_{n=1}^{\infty}{c_n t^{-n}}\in \scrT : c_{n+1} = 0~\text{for all}~n \in I \right\}.
	\end{align*}
	
	\noindent The set $C_0$ is clearly infinite and $\innprod{x}{a_{F; (c_i)}} = 1$ for $x \in C_0$.
	Thus, any enumeration of $\{ a_{F; (c_i)} : \es \ne F \subseteq I~\text{finite}, (c_i) \in \F_p^F \}$
	is a rigidity sequence in $\scrZ$.
\end{example}


\section{Recurrence in abelian groups} \label{sec: recurrence}

Recall the definition of a set of recurrence.

\SetOfRec*

Specifying a bound on the measure of $A$ leads to a related notion.

\begin{definition}
	A set $R \subseteq \Gamma$ is a set of $\delta$-\emph{recurrence} if
	for every measure-preserving system $(X, \B, \mu, (T_g)_{g \in \Gamma})$
	and every $A \in \B$ with $\mu(A) \ge \delta$, there is an $r \in R \setminus \{0\}$ such that
	$\mu(A \cap T_r^{-1}A) > 0$.
\end{definition}

Poincar\'{e}'s Recurrence Theorem can be stated as the fact that $\Z$ is a set of recurrence (in $\Z$).
Analyzing the standard proof, it is not hard to see that $\left\{ 1, 2, \dots, \ceil{\frac{1}{\delta}} \right\}$
is a set of $\delta$-recurrence for every $\delta > 0$.
That is, once $\delta > 0$ is fixed, recurrence is guaranteed to happen in bounded time.
This phenomenon is very robust.
Forrest extended this observation into a fact about general countable discrete groups\footnote{Forrest's proof
does not require that $\Gamma$ be abelian or even amenable.
By imposing this extra condition, we are able to give a proof that
avoids functional analysis in favor of combinatorics.
}
using a clever compactness argument:

\begin{theorem}[Uniformity of Recurrence \cite{forrest}, Lemma 6.4] \label{thm: unif rec}
	A set $R \subseteq \Gamma$ is a set of recurrence if and only if
	for all $\delta > 0$, there is a finite subset $R_{\delta} \subseteq R$
	such that $R_{\delta}$ is a set of $\delta$-recurrence.
\end{theorem}

The proof in \cite{forrest} uses functional analytic techniques.
A simpler proof for $\Z$ is given in \cite{bh}.
We give a new proof of uniformity of recurrence for abelian groups in the spirit of the proof from \cite{bh}.
This requires a discussion of amenability and the introduction of a special class of F{\o}lner sequences,
which we term \emph{self-tiling} F{\o}lner sequences.


\subsection{Upper Banach density}

It is well-known that every abelian group is\break \emph{amenable}.
Roughly speaking, this means that abelian groups have translation-invariant averaging schemes.
We will need some basic definitions and properties relating to amenability
in order to prove Theorem \ref{thm: unif rec}.
Some of these will also appear directly in the construction of rigid-recurrent sequences.

\begin{definition}
	An \emph{invariant mean} is a linear functional $m : \ell^{\infty}(\Gamma) \to \C$ such that
	\begin{enumerate}[(1)]
		\item	$m(\ind) = 1$;
		\item	$m(\tau_t \cdot f) = m(f)$ for every $t \in \Gamma$ and $f \in \ell^{\infty}(\Gamma)$,
			where $(\tau_t \cdot f)(x) := f(x-t)$.
	\end{enumerate}
\end{definition}

The following proposition is a precise statement of the fact that every abelian group is amenable:
\begin{proposition}
	Every abelian group admits an invariant mean.
\end{proposition}

Another characterization of amenability that we will make use of is the existence of a F{\o}lner sequence.
Recall that a sequence $(\Phi_N)_{N \in \N}$ of finite subsets of $\Gamma$ is a F{\o}lner sequence
if for every $x \in \Gamma$,
\begin{align*}
	\lim_{N \to \infty}{\frac{|(\Phi_N+x) \triangle \Phi_N|}{|\Phi_N|}} = 0.
\end{align*}

\noindent That is, the sets $(\Phi_N)_{N \in \N}$ are asymptotically invariant.
Averages along F{\o}lner sequences are closely related to invariant means,
and both can be used to define notions of density.

We define the \emph{upper density} of a set $E$ along a F{\o}lner sequence $\Phi$ by
\begin{align*}
	\overline{d}_{\Phi}(E) := \limsup_{N \to \infty}{\frac{|E \cap \Phi_N|}{|\Phi_N|}}.
\end{align*}

\noindent The \emph{upper Banach density} of a set $E \subseteq \Gamma$ is defined by
\begin{align*}
	d^*(E) := \sup_{\Phi}{\overline{d}_{\Phi}(E)}.
\end{align*}

\noindent This is equivalent to $d^*(E) = \sup_{m \in M(\Gamma)}{m(\ind_E)}$,
where $M(\Gamma)$ is the set of invariant means on $\ell^{\infty}(\Gamma)$.

A surprising characterization of upper Banach density that will be especially useful for us
is given by the following lemma:

\begin{lemma}[\cite{berg-glass}, Lemma 2.7] \label{lem: upper Banach density}
	Let $E \subseteq \Gamma$.
	Then
		\begin{align*}
		d^*(E) = \max \left\{ \alpha \ge 0 : \right. & \text{for every finite set}~F \subseteq \Gamma, \\
		& \left. \text{there is an}~x \in \Gamma~\text{with}~|(E-x) \cap F| \ge \alpha |F| \right\}.
	\end{align*}
	
\end{lemma}


\subsection{Furstenberg correspondence}

\begin{lemma}[Furstenberg Correspondence Principle \cite{et/dp}, Theorem 4.17] \label{lem: correspondence}
	Let $E \subseteq \Gamma$ with $d^*(E) > 0$.
	Then there is a measure-preserving system $(X, \B, \mu, (T_g)_{g \in \Gamma})$ and a set $A \in \B$
	such that $\mu(A) = d^*(E)$ and for every $k \in \N$ and $t_1, \dots, t_k \in \Gamma$, we have
		\begin{align*}
		d^* \left( \bigcap_{i=1}^k{(E-t_i)} \right) \ge \mu \left( \bigcap_{i=1}^k{T_{t_i}^{-1}A} \right).
	\end{align*}
	
\end{lemma}

The inequality in Furstenberg's Correspondence Principle (Lemma \ref{lem: correspondence}) allows one
to deduce combinatorial results from corresponding results about multiple recurrence in dynamical systems.
The following lemma is a useful tool for translating in the other direction: from combinatorics to dynamics.
\begin{lemma}[Bergelson's Intersectivity Lemma \cite{codet}, Lemma 5.10] \label{lem: intersectivity}
	Suppose $(A_g)_{g \in \Gamma}$ is a sequence of sets in a probability space $(X, \B, \mu)$
	with $\mu(A_g) \ge a$ for all $g \in \Gamma$.
	Then for any F{\o}lner sequence $\Phi$ in $\Gamma$, there is a set $E \subseteq \Gamma$
	such that $\overline{d}_{\Phi}(E) \ge a$ and
		\begin{align*}
		\mu \left( \bigcap_{g \in F}{A_g} \right) > 0
	\end{align*}
	
	\noindent for every nonempty finite subset $F \subseteq E$.
\end{lemma}

Using the correspondence principle and the intersectivity lemma together,
one can produce the following characterization of sets of recurrence, which we will utilize in the next section:

\begin{proposition}[\cite{codet}, Theorem 5.13] \label{prop: recurrence difference}
	Suppose $\Gamma$ is abelian, and let $R \subseteq \Gamma$.
	Then $R$ is a set of recurrence if and only if for every $E \subseteq \Gamma$ with $d^*(E) > 0$,
	we have $(E - E) \cap R \ne \es$.
\end{proposition}

The same method can be used to characterize sets of $\delta$-recurrence.
This result is implicit in the proof of uniformity of recurrence for $\Z$ given in \cite{bh}.
We give a proof here in the full generality of abelian groups for completeness.

\begin{proposition} \label{prop: delta-recurrence difference}
	$R$ is a set of $\delta$-recurrence if and only if
	for every $E \subseteq \Gamma$ with $d^*(E) \ge \delta$, we have $(E - E) \cap R \ne \es$.
\end{proposition}
\begin{proof}
	Suppose $R \subseteq \Gamma$ is a set of $\delta$-recurrence, and let $E \subseteq \Gamma$
	with $d^*(E) \ge \delta$.
	By Lemma \ref{lem: correspondence}, let $(X, \B, \mu, (T_g)_{g \in \Gamma})$
	be a measure-preserving system and $A \in \B$ with $\mu(A) = d^*(E) \ge \delta$ such that
		\begin{align} \label{eq: correspondence}
		d^* \left( \bigcap_{i=1}^k{(E - t_i)} \right) \ge \mu\left( \bigcap_{i=1}^k{T_{t_i}^{-1}A} \right)
	\end{align}
	
	\noindent for $k \in \N$ and $t_1, \dots, t_k \in \Gamma$.
	Since $R$ is a set of $\delta$-recurrence, there is an $r \in R$ so that $\mu(A \cap T_r^{-1}A) > 0$.
	It follows from the inequality \eqref{eq: correspondence} that $d^*(E \cap (E - r)) > 0$.
	In particular, $E \cap (E - r) \ne \es$, so let $x \in E \cap (E - r)$.
	Then $x \in E$ and $x + r \in E$, so $r = (x+r) - x \in E - E$.
	Hence, $(E - E) \cap R \ne \es$.
	
	Conversely, suppose $(E - E) \cap R \ne \es$ for every $E \subseteq \Gamma$
	with $d^*(E) \ge \delta$.
	We will show that $R$ is a set of recurrence.
	Let $(X, \B, \mu, (T_g)_{g \in \Gamma})$ be a measure-preserving system,
	and let $A \in \B$ with $\mu(A) \ge \delta$.
	Set $A_g := T_g^{-1}A$ for $g \in \Gamma$.
	Then $(A_g)_{g \in \Gamma}$ is a sequence in $(X, \B, \mu)$ with $\mu(A_g) = \mu(A) \ge \delta$.
	for every $g \in \Gamma$.
	Thus, by Lemma \ref{lem: intersectivity}, there is a set $E \subseteq \Gamma$ with $d^*(E) \ge \delta$
	such that
		\begin{align} \label{eq: intersection}
		\mu \left( \bigcap_{g \in F}{A_g} \right) > 0
	\end{align}
	
	\noindent for every nonempty finite subset $F \subseteq E$.
	Now let $r \in R \cap (E - E)$.
	Let $s, t \in E$ such that $t - s = r$.
	By \eqref{eq: intersection},
		\begin{align*}
		\mu \left( A \cap T_r^{-1}A \right) = \mu \left( A \cap T_{t-s}^{-1}A \right)
		 = \mu \left( T_s^{-1}A \cap T_t^{-1}A \right) = \mu \left( A_s \cap A_t \right) > 0.
	\end{align*}
	
	Thus, $R$ is a set of recurrence.
\end{proof}


\subsection{Tilings}

The last ingredient needed for our proof of Theorem \ref{thm: unif rec}
is a special class of F{\o}lner sequences that we call \emph{self-tiling}.
First we need a round of definitions
(this terminology comes from \cite{ow} in establishing a version of Rokhlin's lemma for amenable groups).

\begin{definition}
	A set $T \subseteq \Gamma$ is a \emph{tile} if there are disjoint translates $\{T + s_i : i \in I\}$
	such that $\bigcup_{i \in I}{(T + s_i)} = \Gamma$.
\end{definition}

\begin{definition} \label{defn: almost invariant}
	Let $K \subseteq \Gamma$ be finite and $\ep > 0$.
	A set $F \subseteq \Gamma$ is \emph{$(K, \ep)$-invariant} if
		\begin{align*}
		\frac{|(K + F) \triangle F|}{|F|} < \ep.
	\end{align*}
	
\end{definition}

\begin{definition}
	A group $\Gamma$ is \emph{monotileable} if for every finite set $K \subseteq \Gamma$
	and every $\ep > 0$, there is a $(K, \ep)$-invariant finite tile $T \subseteq \Gamma$.
\end{definition}

\begin{remark}
	The properties of $T$ in the definition for monotileability can be summarized as follows:
	there is a (syndetic) set $S \subseteq \Gamma$ such that
	
	\begin{enumerate}[(1)]
		\item	$T + S = \Gamma$;
		\item	$(T - T) \cap (S - S) = \{0\}$;
		\item	$\frac{|(K + T) \triangle T|}{|T|} < \ep$.
	\end{enumerate}
	
\end{remark}

It is known that all abelian groups are monotileable (see \cite{ow}).
Other examples of monotileable groups are given in \cite{weiss}.
It remains an open question whether every amenable group is monotileable.

For our purposes, we will need a slightly different notion of almost-invariance than that
given in Definition \ref{defn: almost invariant}, which the following lemma grants:

\begin{lemma}[\cite{dhz}, Lemma 2.6] \label{lem: invariance}
	Let $K \subseteq \Gamma$ be finite and $\ep > 0$.
	If $F \subseteq \Gamma$ is finite and $\left( K, \frac{\ep}{|K|} \right)$-invariant,
	then there is a subset $F' \subseteq F$ such that $|F'| > (1 - \ep)|F|$ and $K + F' \subseteq F$.
\end{lemma}

We also need to control the density of the set of shifts required to tile $\Gamma$.
This is easily calculated with the help of Lemma \ref{lem: upper Banach density}:

\begin{lemma} \label{lem: tile density}
	Suppose $T$ is a tile with shifts $S$ covering $\Gamma$.
	Then $d^*(S) = \frac{1}{|T|}$.
\end{lemma}
\begin{proof}
	First, for every $x \in \Gamma$, $|(S - x) \cap T| \le 1$.
	Indeed, if $s_1 - x = t_1$ and $s_2 - x = t_2$, then $s_1 - s_2 = t_1 - t_2$.
	Since $(T - T) \cap (S - S) = \{0\}$, it follows that $s_1 = s_2$ and $t_1 = t_2$.
	So by Lemma \ref{lem: upper Banach density}, $d^*(S) \le \frac{1}{|T|}$.
	
	Conversely, let $F \subseteq \Gamma$ be finite.
	Every element of $F$ is (uniquely) expressible in the form $t + s$ for some $t \in T$ and $s \in S$.
	For concreteness, write $F = \{t_1 + s_1, \dots, t_n + s_n\}$.
	By the pigeonhole principle, there is a $t_0 \in T$ such that
		\begin{align*}
		|\{1 \le i \le n : t_i = t_0\}| \ge \frac{n}{|T|}.
	\end{align*}
	
	\noindent Thus, $|(S + t_0) \cap F| \ge \frac{1}{|T|} \cdot |F|$.
	Hence, $d^*(S) \ge \frac{1}{|T|}$ by Lemma \ref{lem: upper Banach density}.
\end{proof}

Now we can construct \emph{self-tiling} F{\o}lner sequences.

\begin{lemma} \label{lem: self-tiling}
	Suppose $\Gamma$ is monotileable, and let $(\ep_N)_{N \in \N}$
	be a sequence of positive numbers with $\ep_N \to 0$.
	Then there is a F{\o}lner sequence $(\Phi_N)_{N \in \N}$ such that
	
	\begin{enumerate}[(1)]
		\item	for every $N \in \N$, $\Phi_N$ is a tile;
		\item	for every $N \in \N$ and every $1 \le j \le N-1$,
			there are elements $x_1, \dots, x_r \in \Gamma$ such that $\Phi_j + x_i$ are pairwise disjoint,
			$\bigcup_{i=1}^r{(\Phi_j + x_i)} \subseteq \Phi_N$, and
						\begin{align*}
				\frac{\left| \Phi_N \setminus \bigcup_{i=1}^r{(\Phi_j + x_i)} \right|}
				 {|\Phi_N|} < \ep_N.
			\end{align*}
			
	\end{enumerate}
	
\end{lemma}
\begin{proof}
	Let $g_1, g_2, \dots$ be an enumeration of $\Gamma$.
	Suppose $\Phi_1, \dots, \Phi_{N-1}$ have been constructed to satisfy properties (1) and (2),
	and so that $\Phi_j$ is $\left( \{g_1, \dots, g_j\}, \frac{\ep_j}{j} \right)$-invariant for $1 \le j \le N-1$.
	This invariance is to ensure that the sequence we construct is a F{\o}lner sequence.
	
	Let $K = \bigcup_{j=1}^{N-1}{\Phi_j} \cup \{g_1, \dots, g_N\}$.
	Since $\Gamma$ is monotileable, let $\Phi_N$ be a $\left( K, \frac{\ep_N}{|K|} \right)$-invariant tile.
	By Lemma \ref{lem: invariance}, let $F \subseteq \Phi_N$ with $|F| > (1 - \ep_N) |\Phi_N|$
	such that $K + F \subseteq \Phi_N$.
	Fix $1 \le j \le N-1$.
	Since $\Phi_j$ is a tile, let $S \subseteq \Gamma$ so that $\Phi_j + S = \Gamma$
	and $(\Phi_j - \Phi_j) \cap (S - S) = \{0\}$.
	By Lemma \ref{lem: tile density}, $d^*(S) = \frac{1}{|\Phi_j|}$.
	
	Thus, replacing $S$ by a shift if necessary, we may assume
		\begin{align*}
		|S \cap F| \ge \frac{|F|}{|\Phi_j|} > (1 - \ep_N) \frac{|\Phi_N|}{|\Phi_j|}
	\end{align*}
	
	\noindent by Lemma \ref{lem: upper Banach density}.
	Write $S \cap F = \{x_1, \dots, x_r\}$, and let
		\begin{align*}
		A := \bigcup_{i=1}^r{(\Phi_j + x_i)} = \Phi_j + (S \cap F).
	\end{align*}
	
	\noindent Since this is a disjoint union, we have $|A| = r |\Phi_j| > (1 - \ep_N) |\Phi_N|$.
	Moreover, $A \subseteq \Phi_j + F \subseteq K + F \subseteq \Phi_N$.
	Thus, $\Phi_N$ satisfies property (2).
\end{proof}

\begin{definition}
	A F{\o}lner sequence $(\Phi_N)_{N \in \N}$ is \emph{$(\ep_N)$-self-tiling}
	if it satisfies the conclusion of Lemma \ref{lem: self-tiling}.
\end{definition}

\begin{remark}
	In $\Z$, it is straightforward to construct a self-tiling F{\o}lner sequence.
	For example, the sequence of exponentially growing intervals $\Phi_N = \{1, \dots, 2^N\}$ is self-tiling,
	and without any error $\ep_N$.
\end{remark}


\subsection{Proof of uniformity of recurrence} \label{sec: unif rec proof}

Suppose $R$ is a set of recurrence.
Let $\ep_N \to 0$, and let $(\Phi_N)_{N \in \N}$ be an $(\ep_N)$-self-tiling F{\o}lner sequence.

\begin{claim} \label{claim: unif rec}
	For all $\delta > 0$, there is an $L = L(\delta) \in \N$ such that
	if $N \ge L$ and $F \subseteq \Phi_N$ with $|F| \ge \delta |\Phi_N|$, then $(F - F) \cap R \ne \es$.
\end{claim}
\begin{proof}[Proof of Claim]
	Suppose not.
	For $\alpha > 0$, set
		\begin{align*}
		\mathcal{F}_N(\alpha)
		 := \{F \subseteq \Phi_N + x : x \in \Gamma, |F| \ge \alpha |\Phi_N|, (F - F) \cap R = \es\}.
	\end{align*}
	
	\noindent Note that $\mathcal{F}_N(\alpha)$ is shift-invariant,
	and $\mathcal{F}_N(\alpha) \subseteq \mathcal{F}_N(\beta)$ for $\beta \le \alpha$.
	By assumption, $\mathcal{F}_N(\delta) \ne \es$ for infinitely many $N \in \N$.
	Taking a subsequence of $\Phi$, we may assume that $\mathcal{F}_N(\delta) \ne \es$
	for every $N \in \N$.
	
	Let $0 < \alpha < \delta$.
	Taking a further subsequence, we may assume that $\ep_N$ satisfies
	$\sum_{N=1}^{\infty}{\ep_N} < \delta - \alpha$.
	Let $t_N = \sum_{k=1}^N{\ep_k}$ so that $\alpha < \alpha + t_N < \delta$.
	
	Suppose $F \in \mathcal{F}_N(\alpha + t_N)$.
	Since $\Phi$ is $(\ep_N)$-self-tiling, there are elements $x_1, \dots, x_r\break \in \Gamma$
	such that $\Phi_{N-1} + x_i$ are pairwise disjoint,
	$\bigcup_{i=1}^r{(\Phi_{N-1} + x_i)} \subseteq \Phi_N$, and
		\begin{align*}
		\frac{\left| \Phi_N \setminus \bigcup_{i=1}^r{(\Phi_{N-1} + x_i)} \right|}{|\Phi_N|} < \ep_N.
	\end{align*}
	
	Now we average to find a shift of $\Phi_{N-1}$ in which $F$ has large density:
		\begin{align*}
		\frac{1}{r} \sum_{i=1}^r{\left| F \cap (\Phi_{N-1} + x_i) \right|}
		 & = \frac{1}{r} \left| F \cap \bigcup_{i=1}^r{(\Phi_{N-1} + x_i)} \right| \\
		 & \ge \frac{1}{r} \left( |F| - \left| \Phi_N \setminus \bigcup_{i=1}^r{(\Phi_{N-1} + x_i)} \right| \right) \\
		 & > \frac{1}{r} \left( |F| - \ep_N |\Phi_N| \right) \\
		 & \ge \frac{1}{r} (\alpha + t_N - \ep_N) |\Phi_N| \\
		 & \ge (\alpha + t_{N-1}) |\Phi_{N-1}|.
	\end{align*}
	
	\noindent Thus, $F_{N-1} := F \cap (\Phi_{N-1} + x_i)$ satisfies
	$|F_{N-1}| > (\alpha + t_{N-1})|\Phi_{N-1}|$ for some $1 \le i \le r$.
	Repeating this, we get a sequence $F_1 \subseteq F_2 \subseteq \cdots \subseteq F_{N-1} \subseteq F$
	with $F_j \in \mathcal{F}_j(\alpha + t_j)$.
	
	Now we define a graph in order to extract an infinite nested sequence.
	For $n \in \N$, let $V_n := \left\{ F \in \mathcal{F}_n(\alpha + t_n) : F \subseteq \Phi_n \right\}$.
	Put $V_0 := \{\es\}$, and let $V := \bigcup_{n \ge 0}{V_n}$.
	We define the edge set by
		\begin{align*}
		E := \left\{ \{F_n, F_{n+1}\} : F_n \in V_n, F_{n+1} \in V_{n+1},
		F_n \subseteq F_{n+1} + x~\text{for some}~x \in \Gamma \right\}.
	\end{align*}
	
	\noindent The calculations in the previous paragraph show that the graph $G = (V, E)$ is connected.
	Moreover, since each of the sets $V_n$ is finite for $n \ge 0$, $G$ is a locally finite\footnote{
	A graph is \emph{locally finite} if every vertex has finite degree.
	That is, each vertex is contained in only finitely many edges.}
	graph.
	Hence, by K\H{o}nig's lemma (see, e.g. \cite[Lemma 8.1.2]{diestel}),
	there is an infinite sequence $(F_n)_{n \ge 0}$
	such that $F_n \in V_n$ and $\{F_n, F_{n+1}\} \in E$ for every $n \ge 0$.
	
	For each $n \in \N$, since $\{F_n, F_{n+1}\} \in E$, there exists $x_n \in \Gamma$
	such that $F_n \subseteq F_{n+1} + x_n$.
	Let $y_0 = y_1 = 0$, and let $y_N := \sum_{n=1}^{N-1}{x_n}$ for $N \ge 2$.
	Put $E_N := F_N + y_N$.
	Then $E_N \subseteq F_{N+1} + y_N + x_N = E_{N+1}$.
	We have thus constructed an infinite sequence $E_1 \subseteq E_2 \subseteq \cdots$
	such that $E_N \in \mathcal{F}_N(\alpha + t_N) \subseteq \mathcal{F}_N(\alpha)$ for every $N \in \N$.
	
	Set $E := \bigcup_{N \in \N}{E_N}$.
	Averaging along the F{\o}lner sequence $\Psi_N := \Phi_N + y_N$,
	we have $d^*(E) \ge \overline{d}_{\Psi}(E) \ge \alpha > 0$.
	On the other hand, $(E - E) \cap R = \es$.
	This contradicts the assumption that $R$ is a set of recurrence
	by Proposition \ref{prop: recurrence difference}.
\end{proof}

Now we deduce uniformity of recurrence from the claim.
Given $\delta > 0$, let $L = L(\delta$) as in Claim \ref{claim: unif rec}.
Set $R_{\delta} := R \cap \left( \Phi_L - \Phi_L \right)$.
We claim $R_{\delta}$ is a set of $\delta$-recurrence.

Let $E \subseteq \Gamma$ with $d^*(E) \ge \delta$.
Then for every $N \in \N$, we can apply Lemma \ref{lem: upper Banach density}
to find $x_N \in \Gamma$ such that $|(E - x_N) \cap \Phi_N| \ge \delta |\Phi_N|$.
Let $F := \left( E - x_L \right) \cap \Phi_L$.
Then $F \subseteq \Phi_L$ and $|F| \ge \delta |\Phi_L|$,
so $(F - F) \cap R \ne \es$ by Claim \ref{claim: unif rec}.
But
\begin{align*}
	(F - F) \cap R
	 \subseteq (E - E) \cap \left( \Phi_L - \Phi_L \right) \cap R
	 = (E - E) \cap R_{\delta}.
\end{align*}

\noindent This proves that $R_{\delta}$ is a set of $\delta$-recurrence
by Proposition \ref{prop: delta-recurrence difference}. \qed


\section{Rigid-recurrent sequences} \label{sec: rigid rec}

We are now ready to deal with rigid-recurrent sequences.
Before turning to the construction, we prove Corollary \ref{cor: WM not DS}
to complete the picture about the interactions between rigidity and weak mixing.


\subsection{Proof of Corollary \ref{cor: WM not DS}}

The key to Corollary \ref{cor: WM not DS} is the following:

\begin{proposition} \label{prop: translates not DS}
	Let $R \subseteq \Gamma$.
	Suppose every translate $R - t$, $t \in \Gamma$, is a set of recurrence.
	Then for every enumeration $(r_n)_{n \in \N}$ of $R$,
	$(r_n)_{n \in \N}$ is not rigid for any ergodic system with discrete spectrum.
\end{proposition}
\begin{proof}
	We use the criterion from Lemma \ref{lem: rigid characters}.
	Let $(r_n)_{n \in \N}$ be an enumeration of $R$.
	Suppose $\chi \in \dual$ such that $\chi(r_n) \to 1$.
	The character $\chi$ induces an action of $\Gamma$ on $\T$ by $T_gz := \chi(g) \cdot z$.
	
	Now let $t \in \Gamma$.
	Since $R - t$ is a set of recurrence, there exists, for every $\ep > 0$, an element $r \in R \setminus \{t\}$
	such that $|\chi(r) - \chi(t)| = |\chi(r - t) - 1| < \ep$.
	Hence, there is a subsequence $(r_{n_k})_{k \in \N}$ such that $\chi(r_{n_k}) \to \chi(t)$.
	It follows that $\chi(t) = 1$.
	
	Thus, we have shown that if $\chi(r_n) \to 1$, then $\chi \equiv 1$.
	By Lemma \ref{lem: rigid characters}(1), this proves that $(r_n)_{n \in \N}$ is not a rigidity sequence
	for any ergodic system with discrete spectrum.
\end{proof}

Proposition \ref{prop: translates not DS}, in conjunction with Theorem \ref{thm: rigid-recurrent},
shows that there is a sequence $(r_n)_{n \in \N}$ such that $(r_n)_{n \in \N}$ is a rigidity sequence
for a weakly mixing system but not for any ergodic system with discrete spectrum.
This proves Corollary \ref{cor: WM not DS}.


\subsection{Construction of rigid-recurrent sequences} \label{sec: rigid-rec construction}

We now prove Theorem \ref{thm: rigid-recurrent} following the method of \cite{griesmer}.

A set $K \subseteq \dual$ is \emph{Kronecker} if for every continuous function $f : K \to \T$
and every $\ep > 0$, there is an element $g \in \Gamma$ such that
\begin{align*}
	\sup_{\chi \in K}{\left| f(\chi) - \chi(g) \right|} < \ep.
\end{align*}

\noindent That is, $C(K,\T)$ is the uniform closure of the set $\{e_g : g \in \Gamma\}$,
where $e_g$ is the evaluation map $e_g(\chi) := \chi(g)$.

For a group with finite exponent $n$,
every character takes values in the $n$th roots of unity,
so a nontrivial Kronecker set is too much to hope for.
However, there is a suitable replacement for Kronecker sets in this case.
Denote by $\Lambda_n \subseteq \T$ the group of all $n$th roots of unity for $n \in \N$.
By \cite[Theorem 2.5.5]{rudin}, if $\Gamma$ has finite exponent,
then $\dual$ contains a copy of $D_q := \prod_{n \in \N}{\Lambda_q}$
as a closed subgroup for some $q \in \N$.
A set $K \subseteq D_q \subseteq \dual$ is called a set of \emph{type $K_q$} if for every continuous function
$f : K \to \Lambda_q$, there is an element $g \in \Gamma$ such that $f = e_g|_K$.

First, we show that Kronecker sets and sets of type $K_q$ in $\dual$ can be used to construct sets in $\Gamma$ such that every translate is a set of recurrence:

\begin{proposition} \label{prop: trans rec}
	Let $K \subseteq \dual$ be a Kronecker set or a set of type $K_q$,
	and let $\sigma \in \P(\dual)$ be a probability measure supported on $K$.
	For every $\ep > 0$, every translate of the set
		\begin{align*}
		R_{\ep} := \left\{ g \in \Gamma : \int_\dual{\left| \chi(g) - 1 \right|~d\sigma(\chi)} < \ep \right\}
	\end{align*}
	
	\noindent is a set of recurrence.
\end{proposition}

\begin{corollary} \label{cor: trans rec}
	\begin{enumerate}[(1)]
		\item	If $K$ is a Kronecker set and $\sigma \in \P(\dual)$ is supported on $K$,
			then for every integrable function $f : K \to \T$ and every $\ep > 0$, every translate of
						\begin{align*}
				R_{\ep, f} := \left\{ g \in \Gamma : \int_\dual{\left| \chi(g) - f \right|~d\sigma(\chi)} < \ep \right\}
			\end{align*}
			
			\noindent is a set of recurrence.
		\item	If $K$ is a set of type $K_q$ and $\sigma \in \P(\dual)$ is supported on $K$,
			then for every integrable function $f : K \to \Lambda_q$ and every $\ep > 0$, every translate of
						\begin{align*}
				R_{\ep, f} := \left\{ g \in \Gamma : \int_\dual{\left| \chi(g) - f \right|~d\sigma(\chi)} < \ep \right\}
			\end{align*}
			
			\noindent is a set of recurrence.
	\end{enumerate}
\end{corollary}
\begin{proof}
	Choose $t \in \Gamma$ with $\|e_t - f\|_{L^1(\sigma)} < \frac{\ep}{2}$.
	Then $R_{\ep/2} + t \subseteq R_{\ep, f}$,
	and the result follows from Proposition \ref{prop: trans rec}.
\end{proof}

In order to prove Proposition \ref{prop: trans rec}, we will need the following combinatorial fact:

\begin{lemma}[\cite{griesmer}, Lemma 4.3] \label{lem: cubes}
	Given $\delta > 0$, $\ep > 0$, and $k \in \N$, there is a number $N = N(\delta, \ep, k) \in \N$ such that
	if $d > N$ and $A \subseteq \Lambda_k^d$ has density $\frac{|A|}{|\Lambda_k^d|} \ge \delta$,
	then for every $x \in \Lambda_k^d$, there are $a, b \in A$ such that
	$\rho(ab^{-1}, x) := \frac{1}{d} \sum_{j=1}^d{\left| a_j b_j^{-1} - x_j \right|} < \ep$.
\end{lemma}

Now we can prove Proposition \ref{prop: trans rec}.

\begin{proof}[Proof of Proposition \ref{prop: trans rec}]
	Fix $t \in \Gamma$ and $\ep > 0$.
	We want to show that $R_{\ep} + t$ is a set of recurrence.
	First, note that $s \in R_{\ep} + t$ if and only if
		\begin{align*}
		s - t \in R_{\ep} \iff \int_\dual{\left| \chi(s-t) - 1 \right|~d\sigma(\chi)} < \ep
		 \iff \int_\dual{\left| \chi(s) - \chi(t) \right|~d\sigma(\chi)} < \ep.
	\end{align*}
	
	\noindent Now, by Proposition \ref{prop: recurrence difference}, we can reduce to proving the following:
	for any $E \subseteq \Gamma$ with $d^*(E) > 0$, there exist $a, b \in E$
	such that $\left\| e_{a-b} - e_t \right\|_{L^1(\sigma)} < \ep$.
	
	Now we use a discrete approximation of $e_t$ in order to apply Lemma \ref{lem: cubes}.
	Choose $k \in \N$ sufficiently large that there is a function $\varphi : K \to \Lambda_k$
	with $\left\| \varphi - e_t \right\|_{L^1(\sigma)} < \frac{\ep}{4}$.
	(If $K$ is a set of type $K_q$, take $k = q$.)
	Then choose $d > N \left( d^*(E), \frac{\ep}{4}, k \right)$
	large enough that we can choose such a $\varphi$ to be constant on classes of a partition
	$\P = (P_1, \dots, P_d)$ of $K$ into sets of size $\sigma(P_i) = \frac{1}{d}$.
	
	Consider the set $H_{d,k} :=
	\left\{ \psi = \sum_{j=1}^d{\omega_j \ind_{P_j}} : \omega_1, \dots, \omega_d \in \Lambda_k \right\}$.
	This is a group under pointwise multiplication,
	and the map $\sum_{j=1}^d{\omega_j \ind_{P_j}} \mapsto (\omega_1, \dots, \omega_r)$
	is a group isomorphism $H_{d,k} \simeq \Lambda_k^d$.
	Observe that this group isomorphism is also an isometry
	$(H_{d,k}, \|\cdot\|_{L^1(\sigma)}) \simeq (\Lambda_k^d, \rho)$.
	Indeed, for $\psi = \sum_{j=1}^d{\omega_j \ind_{P_j}}$ and $\psi' = \sum_{j=1}^d{\omega'_j \ind_{P_j}}$
		\begin{align*}
		\left\| \psi - \psi' \right\|_{L^1(\sigma)}
		 & = \int_\dual{\sum_{j=1}^d{\left| \omega_j - \omega'_j \right| \ind_{P_j}}~d\sigma}
		 = \sum_{j=1}^d{\left| \omega_j - \omega'_j \right| \sigma(P_j)} \\
		 & = \frac{1}{d} \sum_{j=1}^d{\left| \omega_j - \omega'_j \right|}
		 = \rho(\omega, \omega').
	\end{align*}
	
	For $\gamma \in \Gamma$, define the sets
		\begin{align*}
		\tilde{E}_{\gamma} := \left\{ \psi \in H_{d,k} :
		 \left\| e_{a + \gamma} - \psi \right\|_{L^1(\sigma)} < \frac{\ep}{4}~\text{for some}~a \in E \right\}
	\end{align*}
	
	\begin{claim} \label{claim: large shift}
		$\left| \tilde{E}_{\gamma} \right| \ge d^*(E) \left| H_{d,k} \right|$
		for some $\gamma \in \Gamma$.
	\end{claim}
	\begin{proof}[Proof of Claim]
		Since $K$ is a Kronecker set or a set of type $K_q$ (with $k = q$),
		we may choose, for each $\psi \in H_{d,k}$,
		an element $a(\psi) \in \Gamma$ such that
		$\left\| e_{a(\psi)} - \psi \right\|_{L^1(\sigma)} < \frac{\ep}{4}$.
		Moreover, we may choose $a(\psi)$ so that $a(\psi) \ne a(\psi')$ for $\psi \ne \psi'$.
		This gives us a set $F_{d,k} := \{ a(\psi) : \psi \in H_{d,k} \}$ with $|F_{d,k}| = |H_{d,k}|$.
		By Lemma \ref{lem: upper Banach density}, there is a $\gamma \in \Gamma$ such that
		$\left| (E+\gamma) \cap F_{d,k} \right| \ge d^*(E) |F_{d,k}|$.
		
		Now, if $a + \gamma \in (E+\gamma) \cap F_{d,k}$,
		then $a + \gamma = a(\psi)$ for some $\psi$,
		so $\left\| e_{a + \gamma} - \psi \right\|_{L^1(\sigma)} < \frac{\ep}{4}$.
		Since the values $a + \gamma = a(\psi)$ are distinct for distinct $\psi \in H_{d,k}$, we have
				\begin{align*}
			\left| \tilde{E}_{\gamma} \right|
			 \ge \left| (E+\gamma) \cap F_{d,k} \right|
			 \ge d^*(E) |F_{d,k}|
			 = d^*(E) |H_{d,k}|.
		\end{align*}
			\end{proof}
	
	Applying Lemma \ref{lem: cubes} to the set $\tilde{E}_{\gamma} \subseteq H_{d,k} \simeq \Lambda_k^d$
	with the tuple $\left( d^*(E), \frac{\ep}{4}, k \right)$,
	we get functions $\psi_1, \psi_2 \in \tilde{E}_{\gamma}$ such that
	$\left\| \psi_1 \overline{\psi}_2 - \varphi \right\|_{L^1(\sigma)} < \frac{\ep}{4}$.
	Thus, we have found $\psi_1, \psi_2 : K \to \T$, $\gamma \in \Gamma$, and $a, b \in E$
	such that:
		\begin{align*}
		\left\| \psi_1 \overline{\psi}_2 - \varphi \right\|_{L^1(\sigma)} & < \frac{\ep}{4}, \\
		\left\| e_{a+\gamma} - \psi_1 \right\|_{L^1(\sigma)} & < \frac{\ep}{4}, \\
		\left\| e_{b+\gamma} - \psi_2 \right\|_{L^1(\sigma)} & < \frac{\ep}{4}.
	\end{align*}
	
	\noindent By the triangle inequality, this yields the desired result:
		\begin{align*}
		\left\| e_{a-b} - e_t \right\|_{L^1(\sigma)} \le
		 & \left\| e_{a+\gamma} \overline{e}_{b+\gamma}
		 - e_{a+\gamma}\overline{\psi}_2 \right\|_{L^1(\sigma)}
		 + \left\| e_{a+\gamma} \overline{\psi}_2- \psi_1\overline{\psi}_2 \right\|_{L^1(\sigma)} \\
		 & + \left\| \psi_1 \overline{\psi}_2 - \varphi \right\|_{L^1(\sigma)}
		 + \left\| \varphi - e_t \right\|_{L^1(\sigma)} < \ep.
	\end{align*}
	\end{proof}

Now that we have found sets all of whose translates are sets of recurrence,
we need to patch them together in such a way as to create a rigidity sequence.
This is where uniformity of recurrence comes into play.
The next lemma will allow us to extract a set of recurrence from a sequence of such sets.

\begin{lemma} \label{lem: seq rec}
	Let $R_1 \supseteq R_2 \supseteq R_3 \supseteq \cdots$ be a descending chain
	of subsets of $\Gamma$ such that every translate is a set of recurrence.
	Then there is a set $R \subseteq \Gamma$ such that $R \setminus R_n$ is finite for every $n \in \N$
	and every translate of $R$ is a set of recurrence.
\end{lemma}
\begin{proof}
	Let $g_1, g_2, \dots$ be an enumeration of $\Gamma$.
	For each $n \in \N$, let $R_n' \subseteq R_n$ be a finite subset such that $R_n' + g_i$
	is a set of $\frac{1}{n}$-recurrence for every $i \le n$ by Theorem \ref{thm: unif rec}.
	Then $R := \bigcup_{n \in \N}{R_n'}$ has the desired properties.
\end{proof}

To tie everything together and finish the proof of Theorem \ref{thm: rigid-recurrent},
we will need a perfect\footnote{Recall
that a set is \emph{perfect} if it is closed and has no isolated points.
}
Kronecker set or a perfect set of type $K_q$.
The following theorem provides this missing piece:

\begin{theorem}[\cite{rudin}, Theorem 5.2.2] \label{thm: perfect K sets}
	\begin{enumerate}[(1)]
		\item	If $\Gamma$ has infinite exponent, then $\dual$ contains a perfect Kronecker set.
		\item	The group $D_q$ contains a perfect set of type $K_q$.
	\end{enumerate}
\end{theorem}

We are now ready to prove Theorem \ref{thm: rigid-recurrent}, restated below.

\RigRec*

\begin{proof}[Proof of Theorem \ref{thm: rigid-recurrent}]
	We prove the case that $\Gamma$ has infinite exponent.
	The finite exponent case holds by the same argument, replacing Kronecker sets by sets of type $K_q$.
	Let $K \subseteq \dual$ be a perfect Kronecker set.
	Then we can choose $(K_g)_{g \in \Gamma}$ mutually disjoint perfect subsets of $K$.
	Every subset of a Kronecker set is clearly Kronecker, so $K_g$ is Kronecker for each $g \in \Gamma$.
	Now for each $g \in \Gamma$, let $\sigma_g$ be a continuous probability measure supported on $K_g$.
	Let $\sigma$ be a weighted average of the measures $\sigma_g$
	so that $\sigma$ is a continuous probability measure on $K$.
	Let $f : K \to \T$ with $f|_{K_g} = e_g$.
	For each $n \in \N$, define
		\begin{align*}
		R_n := \left\{ g \in \Gamma : \int_\dual{\left| e_g - f \right|~d\sigma} < \frac{1}{n} \right\}.
	\end{align*}
	
	\noindent By Corollary \ref{cor: trans rec}, every translate of $R_n$ is a set of recurrence.
	
	Now let $R$ be the set obtained from Lemma \ref{lem: seq rec}.
	Let $(r_k)_{k \in \N}$ be an enumeration of $R$.
	We want to show that every translate of $(r_k)_{k \in \N}$ is a rigidity sequence.
	Let $\ep > 0$.
	Choose $n \in \N$ so that $\frac{1}{n} < \ep$.
	By construction, $R \setminus R_n$ is finite, so there is an $k_0 \in \N$ such that $r_k \in R_n$
	for every $k \ge k_0$.
	Thus, for $k \ge k_0$, $\| e_{r_k} - f \|_{L^1(\sigma)} < \frac{1}{n} < \ep$.
	We have therefore shown that $e_{r_k} \to f$ in $L^1(\sigma)$.
	Restricting to $K_t$, it follows that $e_{r_k} \to e_t$ in $L^1(\sigma_t)$.
	Thus, $e_{r_k - t} \to 1$ in $L^1(\sigma_t)$,
	so $(r_k - t)_{k \in \N}$ is a rigidity sequence for every $t \in \Gamma$.
\end{proof}


\section{Freely rigid-recurrent sequences} \label{sec: free}

In this section, we discuss freely rigid-recurrent sequences
with an eye toward extending Theorem \ref{thm: rigid-recurrent}.


\subsection{Translates of rigidity sequences for free actions} \label{sec: obstacle}

We begin by showing that there is an obstacle to na\"{i}vely extending Theorem \ref{thm: rigid-recurrent}
to free actions.
Following the approach to freeness in the proof of Theorem \ref{thm: free},
one may be inclined to modify the proof of Theorem \ref{thm: rigid-recurrent}
so that the closed subgroup generated by $\supp{\sigma_g}$ is equal to $\dual$ for every $g \in G$.
Unfortunately, this is not possible in general.
To see this, consider the group $\Gamma = (\Z/3\Z) \oplus \bigoplus_{n=1}^{\infty}{(\Z/2\Z)}$,
which has Pontryagin dual $\dual \simeq \Lambda_3 \times \prod_{n=1}^{\infty}{\Lambda_2}$.
Then the type $K_2$ set $K \subseteq \dual$ appearing in the proof of Theorem \ref{thm: free}
sits entirely inside of the closed proper subgroup $\{0\} \oplus \prod_{n=1}^{\infty}{\Lambda_2}$.

The obstacle encountered here is not just an artifact of the method of proof.
For free actions of $(\Z/3\Z) \oplus \bigoplus_{n=1}^{\infty}{(\Z/2\Z)}$,
it is not possible for every translate of a rigidity sequence to be a rigidity sequence for a free action.
This is demonstrated by Corollary \ref{cor: local obstruction} below.

\begin{proposition} \label{prop: local obstruction}
	Let $\Gamma_1$ be a finite group and $\Gamma_2$ a torsion group with finite exponent
	such that the exponents of $\Gamma_1$ and $\Gamma_2$ are coprime.
	If $(a_n)_{n \in \N} = (a_{n,1}, a_{n,2})_{n \in \N}$ is a rigidity sequence
	for a free action of $\Gamma_1 \oplus \Gamma_2$,
	then $a_{n,1} = 0$ for all but finitely many $n \in \N$.
\end{proposition}

\begin{corollary} \label{cor: local obstruction}
	Let $\Gamma_1$ and $\Gamma_2$ be as in Proposition \ref{prop: local obstruction}.
	If $(a_n)_{n \in \N}$ is a rigidity sequence for a free action of $\Gamma_1 \oplus \Gamma_2$,
	then for $g \in \Gamma_1 \setminus \{0\}$,
	$(a_n + (g,0))_{n \in \N}$ is not a rigidity sequence for any free action of $\Gamma_1 \oplus \Gamma_2$.
\end{corollary}

Proposition \ref{prop: local obstruction} follows immediately from the following lemma,
since a sequence in a finite group cannot diverge to infinity:

\begin{lemma} \label{lem: local obstruction}
	Let $\Gamma_1$ and $\Gamma_2$ be groups with finite, coprime exponents.
	If $(a_n)_{n \in \N}$ is a rigidity sequence for an action of $\Gamma_1 \oplus \Gamma_2$,
	then $(a_{n,1}, 0)_{n \in \N}$ is also a rigidity sequence for the action.
	In particular, if the action is free, then
	either $a_{n,1} \to \infty$ or $a_{n,1} = 0$ for all but finitely many $n \in \N$.
\end{lemma}

\begin{proof}
	Let $k$ be the exponent of $\Gamma_1$ and $l$ the exponent of $\Gamma_2$.
	
	Suppose $\left( X, \B, \mu, (T_g)_{g \in \Gamma_1 \oplus \Gamma_2} \right)$
	is a measure-preserving system that is rigid along the sequence $(a_n)_{n \in \N}$.
	Let $f \in L^2(\mu)$.
	We want to show $T_{(a_{n,1},0)}f \to f$ in $L^2(\mu)$.
	
	Without loss of generality, we may assume $\|f\|_2 = 1$.
	Then the spectral measure $\sigma_f$ on $\hat{\Gamma_1 \oplus \Gamma_2}
	 \simeq \hat{\Gamma}_1 \times \hat{\Gamma}_2 =: \dual_1 \times \dual_2$ is a probability measure,
	 and our goal is to show $\hat{\sigma}_f(a_{n,1},0) \to 1$.
	
	By assumption, $\hat{\sigma}_f(a_{n,1}, a_{n,2}) \to 1$.
	Equivalently,
		\begin{align*}
		\int_{\dual_1 \times \dual_2}{\left| \chi_1(a_{n,1}) \chi_2(a_{n,2}) - 1 \right|~d\sigma_f(\chi_1,\chi_2)}
		 \to 0.
	\end{align*}
	
	\noindent This can be rewritten as
		\begin{align*}
		\int_{\dual_1 \times \dual_2}{\left| \chi_1(a_{n,1}) - \overline{\chi_2(a_{n,2})} \right|
		 ~d\sigma_f(\chi_1,\chi_2)}
		 \to 0.
	\end{align*}
	
	\noindent Now, $\chi_1$ takes values in $\Lambda_k$ for every $\chi_1 \in \dual_1$,
	and $\chi_2$ takes values in $\Lambda_l$ for every $\chi_2 \in \dual_2$.
	Since $k$ and $l$ are coprime, we have $\Lambda_k \cap \Lambda_l = \{1\}$.
	Hence, there exists $\delta > 0$ such that if $z \in \Lambda_k$, $w \in \Lambda_l$, and $(z,w) \ne (1,1)$,
	then $|z - w| \ge \delta$.
	Thus,
		\begin{align*}
		\int_{\dual_1 \times \dual_2}{\left| \chi_1(a_{n,1}) - \overline{\chi_2(a_{n,2})} \right|
		 ~d\sigma_f(\chi_1,\chi_2)}
		 \ge \delta \cdot \sigma_f(B),
	\end{align*}
	
	\noindent where $B := \left\{ (\chi_1, \chi_2) \in \dual_1 \times \dual_2
	 : \chi_1(a_{n,1}) \ne 1~\text{or}~\chi_2(a_{n,2}) \ne 1 \right\}$.
	Therefore,
		\begin{align*}
		\left| \hat{\sigma}_f(a_{n,1},0) - 1 \right|
		 & = \left| \int_{\dual_1 \times \dual_2}{\left( \chi_1(a_{n,1}) - 1 \right)
		 ~d\sigma_f(\chi_1,\chi_2)} \right| \\
		 & \le \int_{\dual_1 \times \dual_2}{\left| \chi_1(a_{n,1}) - 1 \right|~d\sigma_f(\chi_1,\chi_2)} \\
		 & \le 2 \sigma_f(B) \\
		 & \le \frac{2}{\delta} \int_{\dual_1 \times \dual_2}{\left| \chi_1(a_{n,1})
		 - \overline{\chi_2(a_{n,2})} \right|~d\sigma_f(\chi_1,\chi_2)} \\
		 & \to 0.
	\end{align*}
	\end{proof}

Note that the obstacle encountered here does not rule out the possibility of producing free actions
if we restrict the shifts to a finite index subgroup.
Concretely, in the example $\Gamma = (\Z/3\Z) \oplus \bigoplus_{n=1}^{\infty}{(\Z/2\Z)}$ discussed above,
it is still possible to produce a sequence such that each shift coming from the finite index subgroup
$\Delta = \bigoplus_{n=1}^{\infty}{(\Z/2\Z)}$ is freely rigid-recurrent.
This is the content of Theorem \ref{thm: free fin exp}
for the group $\Gamma = (\Z/3\Z) \oplus \bigoplus_{n=1}^{\infty}{(\Z/2\Z)}$,
which we will prove in Section \ref{sec: finite exponent}.


\subsection{Topologically-generating Kronecker sets} \label{sec: dense Kronecker}

Analyzing the proof of Theorem \ref{thm: rigid-recurrent}, we see that, for the resulting sequence $(r_n)_{n \in \N}$,
the translate $(r_n - t)_{n \in \N}$ will be freely rigid-recurrent if
the Gaussian system associated to the measure $\sigma_t$ (as defined in the proof) is free.
By Lemma \ref{lem: free WM}, the Gaussian system associated to $\sigma_t$ is free if and only if
$\supp{\sigma_t} = K_t$ generates a dense subgroup of $\dual$.
This motivates the following definition:

\begin{definition} \label{def: Kronecker gen}
	Let $G$ be a compact metrizable abelian group such that either $G$ has infinite exponent
	or $G$ is isomorphic to $D_q$ for some $q \ge 2$.
	We say that $G$ is \emph{topologically Kronecker-generated} if there is a perfect set $K \subseteq G$
	such that $K$ is a Kroncker set or a set of type $K_q$ and the subgroup $\left\langle K \right\rangle$
	is dense in $G$.
\end{definition}

In the proof of Theorem \ref{thm: rigid-recurrent}, we took disjoint perfect sets $(K_g)_{g \in \Gamma}$
and defined $\sigma_g$ to be a continuous measure supported on $K_g$ for each $g \in \Gamma$.
In the process, we may lose control over the group generated by $K_g$, so we take a slightly different approach here.
Given a perfect set $K$ such that $K$ is Kronecker or of type $K_q$ and $\left\langle K \right\rangle$ is dense in $\dual$,
we can instead take $(\sigma_g)_{g \in \Gamma}$ to be a collection of mutually singular continuous measures
with $\supp{\sigma_g} = K$ for every $g \in \Gamma$.
Then choosing mutually disjoint Borel sets $(K_g)_{g \in \Gamma}$ with $\sigma_g(K_g) = 1$,
we may repeat the remainder of the proof of Theorem \ref{thm: rigid-recurrent}
to construct a sequence, all of whose translates are freely rigid-recurrent sequences.

Hence, for groups with infinite exponent, we have the following conditional result:

\begin{proposition}
	Let $\Gamma$ be a countable discrete abelian group with infinite exponent, and let $\dual = \hat{\Gamma}$.
	If $\dual$ is topologically Kronecker-generated, then there is a sequence $(r_n)_{n \in \N}$ in $\Gamma$
	such that, for every $t \in \Gamma$, $(r_n - t)_{n \in \N}$ is a freely rigid-recurrent sequence.
\end{proposition}

We have thereby reduced the infinite exponent case of Conjecture \ref{conj: free rigid-recurrent}
to the following harmonic analytic problem:

\begin{question} \label{q: dense Kronecker}
	Let $G$ be a compact metrizable abelian group with infinite exponent.
	Is $G$ topologically Kronecker-generated?
\end{question}

For groups with finite exponent, additional technical results are needed to piece together sets of type $K_q$
for different values of $q$.
Nevertheless, we will show that if $D_q$ is topologically Kronecker-generated for every $q \ge 2$,
then Conjecture \ref{conj: free rigid-recurrent} holds for every countable discrete abelian group $\Gamma$
with finite exponent.
The details of this argument are presented in Section \ref{sec: finite exponent} below
(see, in particular, Theorem \ref{thm: Delta shifts}).
Additionally, we will show that $D_p$ is topologically Kronecker-generated for prime $p$
and deduce from this Theorem \ref{thm: free fin exp}.


\subsection{Groups with finite exponent} \label{sec: finite exponent}

Throughout this section, we make the following assumptions.
We fix finitely many prime powers $2 \le q_1 < q_2 < \dots < q_N$.
We then let $\Gamma$ be the group
\begin{align*}
	\Gamma := \bigoplus_{j=1}^N{\left( \bigoplus_{n=1}^{\infty}{(\Z/q_j\Z)} \right)}.
\end{align*}

\noindent As in Section \ref{sec: rigid-rec construction}, we let $D_q := \prod_{n=1}^{\infty}{\Lambda_q}$
so that $\Gamma$ has dual group
\begin{align*}
	\dual = \prod_{j=1}^N{D_{q_j}}.
\end{align*}

\noindent For each $j = 1, \dots, N$, we let $i_j : \bigoplus_{n=1}^{\infty}{(\Z/q_j\Z)} \to \Gamma$
and $\iota_j : D_{q_j} \to \dual$ be the natural inclusions.
That is,
\begin{align*}
	i_j(g) & = (0, \dots, 0, g, 0, \dots, 0), \\
	\iota_j(\chi) & = (1, \cdots, 1, \chi, 1, \cdots, 1),
\end{align*}

\noindent where the nontrivial entry is in the $j$th coordinate.

The following lemma allows us to use tools available for sets of type $K_q$ when there is non-uniform torsion.

\begin{lemma} \label{lem: gen Kronecker}
	For each $j = 1, \dots, N$, let $K_j \subseteq D_{q_j}$ be a set of type $K_{q_j}$, and let
	$K := \bigcup_{j=1}^N{\iota_j(K_j)}$.
	If $f : K \to \T$ is continuous and $f \left( \iota_j(K_j) \right) \subseteq \Lambda_{q_j}$,
	then there exists an element $g \in \Gamma$ such that $f = \left. e_g \right|_K$.
\end{lemma}
\begin{proof}
	For each $j = 1, \dots, N$, define $f_j : K_j \to \T$ by $f_j = f \circ \iota_j$.
	Note that $f_j$ takes values in $\Lambda_{q_j}$.
	Therefore, since $K_j$ is a set of type $K_{q_j}$, there is an element $g_j \in \bigoplus_{n=1}^{\infty}{(\Z/q_j\Z)}$
	such that $f_j(\chi) = \chi(g_j)$ for every $\chi \in K_j$.
	
	Let $g = \sum_{j=1}^N{i_j(g_j)} \in \Gamma$.
	Then for $\chi \in K_j$, we have
		\begin{align*}
		\left( \iota_j(\chi) \right)(g)
		 & = \left( 1, \cdots, 1, \chi, 1, \cdots, 1 \right)
		 (g_1, \dots, g_{j-1}, g_j, g_{j+1}, \dots, g_N) \\
		 & = \chi(g_j)
		 = f_j(\chi)
		 = f\left( \iota_j(\chi) \right).
	\end{align*}
	
	\noindent Thus, $f = e_g$ on $\iota_j(K_j)$ for each $j$, so this equality holds on their union $K$.
\end{proof}

\begin{proposition} \label{prop: gen trans rec}
	For each $j = 1, \dots, N$, let $K_j \subseteq D_{q_j}$ be a set of type $K_{q_j}$, and let
	$K := \bigcup_{j=1}^N{\iota_j(K_j)}$.
	Let $\sigma \in \P(\dual)$ be a probability measure supported on $K$.
	For every $\ep > 0$, every translate of the set
		\begin{align*}
		R_{\ep} := \left\{ g \in \Gamma : \int_\dual{\left| \chi(g) - 1 \right|~d\sigma(\chi)} < \ep \right\}
	\end{align*}
	
	\noindent is a set of recurrence.
\end{proposition}

Applying Lemma \ref{lem: gen Kronecker}, we immediately get the following
as a corollary of Proposition \ref{prop: gen trans rec}:

\begin{corollary} \label{cor: gen trans rec}
	Let $\sigma \in \P(\dual)$ be a probability measure supported on $K$.
	Let $f : K \to \T$ be an integrable function such that
	$f \left( \iota_j(K_j) \right) \subseteq \Lambda_{q_j}$ for $j = 1, \dots, N$.
	Then for every $\ep > 0$, every translate of the set
		\begin{align*}
		R_{\ep, f} := \left\{ g \in \Gamma : \int_\dual{\left| \chi(g) - f \right|~d\sigma(\chi)} < \ep \right\}
	\end{align*}
	
	\noindent is a set of recurrence.
\end{corollary}

Before proving Proposition \ref{prop: gen trans rec}, we discuss
how it can be used to construct freely rigid-recurrent sequences.
Assume $D_{q_j}$ is topologically Kronecker generated for every $j = 1, \dots, N$.
For each $j = 1, \dots, N$, let $K_j \subseteq D_{q_j}$ be a perfect Kronecker set
such that $\left\langle K_j \right\rangle$ is dense in $D_{q_j}$.
Then let $(\sigma_{j,g})_{g \in \Gamma}$ be a family of mutually singular continuous probability measures on $\dual$
with $\supp{\sigma_{j,g}} = \iota_j(K_j)$.
Put $\sigma_g := \frac{1}{N} \sum_{j=1}^N{\sigma_{j,g}}$ for $g \in \Gamma$
so that $\supp{\sigma_g} = K := \bigcup_{j=1}^N{\iota_j(K_j)}$.
It is easily checked that
\begin{align*}
	\left\langle K \right\rangle = \prod_{j=1}^N{\left\langle K_j \right\rangle},
\end{align*}

\noindent so $K$ generates a dense subgroup of $\prod_{j=1}^N{D_{q_j}} = \dual$.
Running through the remainder of the proof of
Theorem \ref{thm: rigid-recurrent} with Corollary \ref{cor: gen trans rec} in place of Corollary \ref{cor: trans rec} gives:

\begin{proposition} \label{prop: conditional free torsion}
	If $D_q$ is topologically Kronecker-generated for each $q \in \{q_1, \dots,\break q_N\}$,
	then there is a sequence $(r_n)_{n \in \N}$ in $\Gamma$ such that, for every $t \in \Gamma$,
	$(r_n - t)_{n \in \N}$ is a freely rigid-recurrent sequence.
\end{proposition}

In fact, we can prove the following more general result:

\begin{theorem} \label{thm: Delta shifts}
	Let $\Gamma$ be a countable discrete abelian group with finite exponent, and write
		\begin{align*}
		\Gamma := \bigoplus_{k=1}^M{(\Z/r_k\Z)^{m_k} }
		 \oplus \bigoplus_{j=1}^N{\left( \bigoplus_{n=1}^{\infty}{(\Z/q_j\Z)} \right)}
	\end{align*}
	
	\noindent with $q_1, \dots, q_N, r_1, \dots, r_M$ distinct prime powers and $m_1, \dots, m_M \in \N$.
	Let $\Delta \le \Gamma$ be the finite index subgroup
		\begin{align*}
		\Delta := \bigoplus_{j=1}^N{\left( \bigoplus_{n=1}^{\infty}{(\Z/q_j\Z)} \right)}.
	\end{align*}
	
	\noindent If $D_q$ is topologically Kronecker-generated for each $q \in \{q_1, \dots, q_N\}$,
	then there is a sequence $(r_n)_{n \in \N}$ in $\Gamma$ such that, for every $s \in \Delta$,
	$(r_n - s)_{n \in \N}$ is a freely rigid-recurrent sequence.
\end{theorem}

\begin{remark} \label{rem: Delta'}
	The group $\Delta$ in Theorem \ref{thm: Delta shifts} is not optimal in all situations,
	but it is nearly so.
	In particular, it is possible for there to be a group $\Delta < \Delta' \le \Gamma$
	with $\Delta' \simeq \Delta$, in which case the same result will hold for $\Delta'$.
	This happens if $q_j \mid r_k$ for some $j, k$.
	
	As an illustration, consider the group $\Gamma = (\Z/4\Z) \oplus \left( \bigoplus_{n=1}^{\infty}{(\Z/2\Z)} \right)$,
	and write elements of $\Gamma$ as $g = (g_0; g_1, g_2, \dots)$
	with $g_0 \in \Z/4\Z$ and $g_n \in \Z/2\Z$ for $n \ge 1$.
	Then $\Delta = \{g \in \Gamma : g_0 = 0\}$, which is an index 4 subgroup.
	However, we could instead take $\Delta' = \{g \in \Gamma : g_0 \in \{0,2\}\}$,
	which contains $\Delta$ and has index 2.
	
	The argument in the proof of Lemma \ref{lem: local obstruction}
	can be adapted to show that $\Delta'$ is optimal.
	Namely, if $(r_n)_{n \in \N}$ is rigid for a free action of $\Gamma$,
	then $r_{n,0} \equiv 0 \pmod{2}$ for all but finitely many $n \in \N$.
\end{remark}

\begin{proof}[Proof of Theorem \ref{thm: Delta shifts}]
	Let $\dual_0 := \hat{\Delta}$.
	Then $\dual_0 \simeq \dual/\Delta^{\perp}$,
	where $\Delta^{\perp} := \{\chi \in \dual : \chi(s) = 1~\text{for all}~s \in \Delta\}$
	is the annihilator of $\Delta$.
	Since $\Delta \le \Gamma$ has finite index, $\Delta^{\perp} \simeq \hat{\Gamma/\Delta}$ is finite.
	
	Apply Proposition \ref{prop: conditional free torsion} to $(\Delta, \dual_0)$
	to obtain a sequence $(r_n)_{n \in \N}$ in $\Delta$.
	We claim that $(r_n)_{n \in \N}$ has the desired properties (now considered as a sequence in $\Gamma$).
	
	Let $s \in \Delta$.
	Clearly $\{r_n - s : n \in \N\}$, being a set of recurrence in $\Delta$,
	is also a set of recurrence in $\Gamma$.
	We now verify that $(r_n - s)_{n \in \N}$ is a rigidity sequence for a free action of $\Gamma$.
	By construction and Lemmas \ref{lem: rigid characters}(2) and \ref{lem: free WM},
	there is a continuous probability measure $\sigma_{0,s} \in \P_c(\dual_0)$
	such that $\hat{\sigma}_{0,s}(r_n) \to 1$ as $n \to \infty$
	and $G(\sigma_{0,s}) := \overline{\left\langle \supp{\sigma_{0,s}} \right\rangle} = \dual_0$.
	Define $\sigma_s \in \P_c(\dual)$ via $\sigma_s(A) = \sigma_{0,s}(\pi(A))$,
	where $\pi : \dual \to \dual_0$ is the projection map.
	Then $G(\sigma_s) := \overline{\left\langle \supp{\sigma_{s}} \right\rangle} = \dual$, and
		\begin{align*}
		\hat{\sigma}_s(r_n - s) & = \int_{\dual}{\chi(r_n-s)~d\sigma_s(\chi)} \\
		 & = \frac{1}{|\Delta^{\perp}|} \sum_{\chi_1 \in \Delta^{\perp}}
		 {\int_{\dual_0}{\chi_0(r_n - s) \chi_1(r_n - s)~d\sigma_{0,s}(\chi_0)}} \\
		 & = \int_{\dual_0}{\chi_0(r_n- s)~d\sigma_{0,s}(\chi_0)} \\
		 & = \hat{\sigma}_{0,s}(r_n - s) \to 1.
	\end{align*}
	
	\noindent By Lemma \ref{lem: rigidity Gaussian}, the Gaussian system associated to $\sigma_s$
	is rigid along the sequence $(r_n - s)_{n \in \N}$.
	Moreover, by Lemma \ref{lem: free WM}, this system is free.
\end{proof}

By Theorem \ref{thm: Delta shifts}, all that remains to prove Theorem \ref{thm: free fin exp} is the following observation:

\begin{lemma} \label{lem: finite exp Kronecker gen}
	For any prime number $p$, the group $D_p$ is topologically Kronecker-generated.
\end{lemma}
\begin{proof}
	By Theorem \ref{thm: perfect K sets}, let $K \subseteq D_p$ be a perfect set of type $K_p$.
	Then the closed subgroup $G$ generated by $K$ must be isomorphic to $D_p$.
	Indeed, every infinite closed subgroup of $D_p$ is isomorphic to $D_p$.
	Letting $\varphi : G \to D_p$ be a homeomorphic isomorphism,
	it is easy to check that $\varphi(K)$ is a perfect set of type $K_p$ generating a dense subgroup of $D_p$.
\end{proof}

Now we proceed to prove Proposition \ref{prop: gen trans rec}.
We begin the same way as the proof of Proposition \ref{prop: trans rec}.
Fix an element $t \in \Gamma$ and a set $E \subseteq \Gamma$ of positive density $d^*(E) = \delta > 0$.
We want to find $a, b \in E$ such that $\|e_{a-b} - e_t\|_{L^1(\sigma)} < \ep$.

Discarding a negligible portion of $K$, we may partition, for some large $M \in \N$,
each of the sets $\iota_j(K_j)$ into cells of $\sigma$-measure $\frac{1}{M}$.
Thus, we may assume that $K = \bigcup_{l=1}^M{\tilde{K}_l}$ with the property that, for each $l = 1, \dots, M$:

\begin{enumerate}[(1)]
	\item	there exists $k_l \in \N$ such that every element of $\tilde{K}_l$ has order $k_l$,
		and $\tilde{K}_l$ is a set of type $K_{k_l}$; and
	\item	$\sigma \left( \tilde{K}_l \right) = \frac{1}{M}$.
\end{enumerate}

Now define $\Lambda = \prod_{l=1}^M{\Lambda_{k_l}}$,
and choose $d > N\left( \delta, \frac{\ep}{4}, \Lambda \right)$, to be defined later, large enough
that there are functions $\varphi_l : \tilde{K}_l \to \Lambda_{k_l}$ such that
$\int_{\tilde{K}_l}{|\varphi_l - e_t|~d\sigma} < \frac{\ep}{4M}$ and $\varphi_l$ is constant on classes of a partition
$\P_l = (P_{1,l}, \dots, P_{d,l})$ of $\tilde{K}_l$ into sets of measure $\sigma(P_{j,l}) = \frac{1}{dM}$.
Define $\varphi : K \to \T$ by $\varphi(\chi) = \varphi_l(\chi)$ if $\chi \in \tilde{K}_l$.
Note that $\|\varphi - e_t\|_{L^1(\sigma)} < \frac{\ep}{4}$.

The partitions $\P_l$ allow us to produce a copy of $\Lambda^d$ containing the function $\varphi$.
Let $H :=
 \left\{ \psi = \sum_{j=1}^d{\sum_{l=1}^M{\omega_{j,l}\ind_{P_{j,l}}}} : \omega_{j,l} \in \Lambda_{k_l} \right\}$.
By construction, every element of $H$ satisfies $f \left( \iota_j(\tilde{K}_l) \right) \subseteq \Lambda_{k_l}$.
Moreover, the map
$\psi = \sum_{j=1}^d{\sum_{l=1}^M{\omega_{j,l}\ind_{P_{j,l}}}}
 \mapsto (\omega_{j,l})_{1 \le j \le d, 1 \le l \le M}$
is an isometric isomorphism of the metric groups $(H, \|\cdot\|_{L^1(\sigma)}) \simeq (\Lambda^d, \rho)$,
where
\begin{align*}
	\rho(\omega, \omega') = \frac{1}{dM} \sum_{j=1}^d{\sum_{l=1}^M{\left| \omega_{j,l} - \omega'_{j,l} \right|}}.
\end{align*}

It can be shown that the set
\begin{align*}
	\tilde{E}_{\gamma}
	 := \left\{ \psi \in H : \left\| e_{a + \gamma} - \psi\right\|_{L^1(\sigma)} < \frac{\ep}{4}
	 ~\text{for some}~a \in E \right\}
\end{align*}

\noindent has cardinality at least $\delta |H|$ for some $\gamma \in \Gamma$
(see the proof of Claim \ref{claim: large shift}).

To prove Proposition \ref{prop: gen trans rec}, it therefore suffices to prove
the following generalization of Lemma \ref{lem: cubes}:

\begin{lemma} \label{lem: gen cubes}
	Let $M \in \N$ and $k_1, \dots, k_M \in \N$, and set $\Lambda = \prod_{l=1}^M{\Lambda_{k_l}}$.
	Given $\delta, \ep > 0$, there is a number $N = N(\delta, \ep, \Lambda)$ such that if $d > N$
	and $A \subseteq \Lambda^d$ has density $\frac{|A|}{|\Lambda^d|} \ge \delta$,
	then for every $x \in \Lambda^d$, there are $a, b \in A$ such that
		\begin{align*}
		\frac{1}{dM} \sum_{j=1}^d{\sum_{l=1}^M{\left| a_{j,l}b_{j,l}^{-1} - x_{j,l} \right|}} < \ep.
	\end{align*}
	
\end{lemma}

Indeed, assuming Lemma \ref{lem: gen cubes} and using the isomorphism $H \simeq \Lambda^d$,
we get functions $\psi_1, \psi_2 \in \tilde{E}_{\gamma}$
such that $\|\psi_1 \overline{\psi_2} - \varphi\|_{L^1(\sigma)} < \frac{\ep}{4}$,
and the result follows by an application of the triangle inequality.

To prove Lemma \ref{lem: gen cubes}, we use a result of McDiarmid:

\begin{theorem}[\cite{mcdiarmid}, Theorem 7.9(b)] \label{thm: mcdiarmid}
	Let $G$ be a group with translation-invariant metric $\rho$.
	Let $G = G_0 \supseteq G_1 \supseteq \cdots \supseteq G_n = \{e\}$
	be a decreasing sequence of subgroups,
	and let $c_k$ be the diameter of $G_{k-1}/G_k$ for $k = 1, \dots, n$.
	If $A \subseteq G$ has density $\frac{|A|}{|G|} = \alpha \in (0,1)$, then for any $t > 0$
	and $\gamma > 0$,
		\begin{align*}
		\frac{|A_t|}{|G|} \ge 1 - \left( \frac{1}{\alpha} \right)^{\gamma^2}
		 \exp \left( - \frac{2 \left( \frac{\gamma}{1 + \gamma} \right)^2 t^2}{\sum_{j=1}^n{c_k^2}} \right),
	\end{align*}
	
	\noindent where $A_t = \left\{ x \in G : \rho(x,a) < t~\text{for some}~a \in A \right\}$.
\end{theorem}

\begin{proof}[Proof of Lemma \ref{lem: gen cubes}]
	Fix $d \in \N$.
	Define a translation-invariant metric $\rho$ on $\Lambda^d$ by
		\begin{align*}
		\rho(x,y) := \frac{1}{dM} \sum_{j=1}^d{\sum_{l=1}^M{\left| x_{j,l} - y_{j,l} \right|}}.
	\end{align*}
	
	\noindent We want to show that, if $d$ is sufficiently large,
	then every set $A \subseteq \Lambda^d$ with density at least $\delta$
	satisfies $A \cap (x \cdot A_{\ep}) \ne \es$ for every $x \in \Lambda^d$.
	Indeed, if $a \in A \cap (x \cdot A_{\ep})$, then $a = xbt$ for some $b \in A$ and $t \in \Lambda^d$
	with $\rho(t,1) < \ep$, so
		\begin{align*}
		\rho(ab^{-1}, x) = \rho(a, xb) = \rho(xbt,xb) = \rho(t,1) < \ep.
	\end{align*}
	
	\noindent Thus, it suffices to show that $|A| + |A_{\ep}| > |\Lambda^d|$ for large $d$.
	
	For $k = 1, \dots, d$, let
	$G_k = \{x = (x_1, \dots, x_d) \in \Lambda^d : x_j = 1~\text{for}~j \le k\} \simeq \Lambda^{d-k}$.
	Then
		\begin{align*}
		c_k := \textrm{diam}_{\rho}(G_{k-1}/G_k)
		 = \frac{1}{dM} \sum_{l=1}^M{\textrm{diam}(\Lambda_{k_l})} \le \frac{2}{d}.
	\end{align*}
	
	\noindent Taking $\gamma = 1$ in Theorem \ref{thm: mcdiarmid}, we have
		\begin{align} \label{eq: A_t bound}
		\frac{|A_t|}{|\Lambda^d|} \ge 1 - \alpha^{-1} \exp(-dt^2/8).
	\end{align}
	
	\noindent for subsets $A \subseteq \Lambda^d$, $\alpha = \frac{|A|}{|\Lambda^d|}$ and $t > 0$.
	
	Let $N = \ceil{16 \ep^{-2} \ln{\left( \delta^{-1} \right)}}$.
	For $d > N$, if $A \subseteq \Lambda^d$ has density at least $\delta$, then by \eqref{eq: A_t bound},
		\begin{align*}
		\frac{|A| + |A_{\ep}|}{|\Lambda^d|} \ge \delta + 1 - \delta^{-1} \exp(-d\ep^2/8)
		 > \delta + 1 - \delta^{-1} \exp(2\ln{\delta}) = 1.
	\end{align*}
	\end{proof}


\end{document}